\documentclass[11pt]{article}
\usepackage[centertags]{amsmath}
\usepackage{amsfonts}
\usepackage{amssymb}
\usepackage{amsthm,color}
\usepackage{newlfont}
\usepackage{bbm}

\newtheorem{Theorem}{Theorem}[section]

\newtheorem{Lemma}[Theorem]{Lemma}
\newtheorem{Corollary}[Theorem]{Corollary}
\newtheorem{Remark}[Theorem]{Remark}

\newtheorem{Hypothesis}{Hypothesis}

\numberwithin{equation}{section}

\begin{document}

\def\le{\left}
\def\r{\right}
\def\cost{\mbox{const}}
\def\a{\alpha}
\def\d{\delta}
\def\ph{\varphi}
\def\e{\epsilon}
\def\la{\lambda}
\def\si{\sigma}
\def\La{\Lambda}
\def\B{{\cal B}}
\def\A{{\mathcal A}}
\def\L{{\mathcal L}}
\def\O{{\mathcal O}}
\def\bO{\overline{{\mathcal O}}}
\def\F{{\mathcal F}}
\def\K{{\mathcal K}}
\def\H{{\mathcal H}}
\def\D{{\mathcal D}}
\def\C{{\mathcal C}}
\def\M{{\mathcal M}}
\def\N{{\mathcal N}}
\def\G{{\mathcal G}}
\def\T{{\mathcal T}}
\def\R{{\mathbb R}}
\def\I{{\mathcal I}}

\def\bw{\overline{W}}
\def\phin{\|\varphi\|_{0}}
\def\s0t{\sup_{t \in [0,T]}}
\def\lt{\lim_{t\rightarrow 0}}
\def\iot{\int_{0}^{t}}
\def\ioi{\int_0^{+\infty}}
\def\ds{\displaystyle}
\def\pag{\vfill\eject}
\def\fine{\par\vfill\supereject\end}
\def\acapo{\hfill\break}

\def\beq{\begin{equation}}
\def\eeq{\end{equation}}
\def\barr{\begin{array}}
\def\earr{\end{array}}
\def\vs{\vspace{.1mm}   \\}
\def\rd{\reals\,^{d}}
\def\rn{\reals\,^{n}}
\def\rr{\reals\,^{r}}
\def\bD{\overline{{\mathcal D}}}
\newcommand{\dimo}{\hfill \break {\bf Proof - }}
\newcommand{\nat}{\mathbb N}
\newcommand{\E}{\mathbb E}
\newcommand{\Pro}{\mathbb P}
\newcommand{\com}{{\scriptstyle \circ}}
\newcommand{\reals}{\mathbb R}

\def\Amu{{A_\mu}}
\def\Qmu{{Q_\mu}}
\def\Smu{{S_\mu}}
\def\H{{\mathcal{H}}}
\def\Im{{\textnormal{Im }}}
\def\Tr{{\textnormal{Tr}}}
\def\E{{\mathbb{E}}}
\def\P{{\mathbb{P}}}
\def\span{{\textnormal{span}}}
\title{On the Smoluchowski-Kramers approximation for a system with  infinite degrees of freedom exposed to a magnetic field}
\author{Sandra Cerrai, Michael Salins\\
\vspace{.1cm}\\
Department of Mathematics\\
 University of Maryland\\
College Park\\
 Maryland, USA
}

\date{}

\maketitle

\begin{abstract}
We study the validity of the so-called Smoluchowski-Kramers approximation for a two dimensional system of stochastic partial differential equations, subject to a constant magnetic field. As the small mass limit does not yield to the solution of the corresponding first order system, we regularize our problem by adding a small friction. We show that in this case the Smoluchowski-Kramers approximation holds. We also give a justification of the regularization, by showing that the regularized problems provide a good approximation to the original ones.
\end{abstract}

\section{Introduction}

We consider here the following two dimensional system of stochastic PDEs

\begin{equation} \label{intro-eq}
\left\{
  \begin{array}{l}
    \ds{\mu \frac{\partial^2 u_{\mu}}{\partial t^2}(\xi,t) = \Delta u_{\mu}(\xi,t) + B(u_\mu(\cdot,t),t)  +\vec{m}\times \frac{\partial u_{\mu}}{\partial t}(\xi,t) + G(u_\mu(\cdot,t),t)\,\frac{\partial  w^{\,Q}}{\partial t}(\xi,t),}\\
    \vs
 \ds{u_\mu(\xi,0) = u_0(\xi), \ \ \ \frac{\partial u_\mu}{\partial t}(\xi,0) = v_0(\xi), \ \ \ \xi \in D,\ \ \ \ \ \ \ u_\mu(\xi,t) = 0, \ \ \xi \in \partial D, }\\
  \end{array} \right.
\end{equation}
where $D$ is a bounded regular domain in $\reals^d$, with $d\geq 1$, $B$ and $G$ are suitable nonlinearities, $\vec{m}=(0,0,m)$ is a constant vector and $w^Q(t,\xi)$ is a cylindrical Wiener process, white in time and colored in space, in the case of space dimension $d>1$.

By Newton's law, the vector field $u_\mu:D\to \reals^2$ models the displacement of a continuum of particles with constant density $\mu>0$ in the region $D \subset \mathbb{R}^d$, in presence of a noisy perturbation and a constant magnetic field $\vec{m}=(0,0,m)$, which is orthogonal to the plane where the motion occurs (in what follows we shall assume just for simplicity of notations $m=1$). For example, if $d=1$ and $D=[0,1]$, this could model the displacement of a one-dimensional string, with fixed endpoints, that can move through two other spacial dimensions, where the Laplacian $\Delta$ models the forces neighboring particles exert on each other, $B$ is some nonlinear forcing, and $\partial w^Q/\partial t$ is a Gaussian random forcing field, whose intensity $G$ may depend on the state $u_\mu$.

In \cite{cf-2006-add} and \cite{cf-2006-mult}, we prove the validity of the  so-called  Smoluchowski-Kramers approximation, in the case the magnetic field is  replaced by a constant friction. Namely, it has been shown that, as $\mu$ tends to $0$, the solutions of the second order system converge to the solution of the first order system which is obtained simply by taking $\mu=0$. Moreover, in \cite {cs1} and \cite{cs2} we have studied the interplay between the Smoluchowski-Kramers approximation and the large deviation principle. In particular, we have shown how some relevant quantities associated with large deviations and exit problems from a basin of attraction for the second order problem can be approximated by the corresponding quantities for the first order problem, in terms of the small mass asymptotics described by the Smoluchowski-Kramers approximation.

\medskip

One might hope that a similar result would be true in the case treated in the present paper. Namely, one would expect that for any $T>0$ and $p\geq 1$
\begin{equation}
\label{lim}
\lim_{\mu\to 0}\,\E\sup_{t \in\,[0,T]}|u_\mu(t)-u(t)|_{ L^2(D;\reals^2)}^p=0,
\end{equation}
where $u(t)$ is the solution of the following system of stochastic PDEs
\begin{equation}
\label{intro-eq2}
\left\{
  \begin{array}{l}
    \ds{\frac{\partial u}{\partial t}(\xi,t) = J_{0}^{-1}\Big[\Delta u(\xi,t) +  B(u(\cdot,t),t) + G(u(\cdot,t),t) \frac{\partial w^{Q}}{\partial t}(\xi,t)\Big] }\\
    \vs
   \ds{u(\xi,t) = 0, \ \ \xi \in \partial D,\ \ \ \ \ \ u(\xi,0) = u_0(\xi),  \xi \in D, }
  \end{array} \right.
\end{equation}
where
\[J_0^{-1}=-\,J_0=\begin{pmatrix} 0 & -1 \\ 1 & 0 \end{pmatrix}.\]

Unfortunately, as shown in \cite{cf-2011} such a limit is not valid, even for finite dimensional analogues of this problem.
Actually, one can prove that if the stochastic term in \eqref{intro-eq} is replaced by a continuous function, then $u_\mu$ would converges uniformly in $[0,T]$ to the solution of \eqref{intro-eq2}. But if we have the white noise term, this is not true anymore. An explanation of this lies in the fact that, while for any continuous function $\varphi(s)$ it holds
\[\lim_{\mu\to 0}\int_0^t\sin (s/\mu)\,\varphi(s)\,ds=0,\]
if we consider a stochastic integral and replace $\varphi(s)ds$ with $dB(s)$, we have
\[\lim_{\mu\to 0}\int_0^t\sin (s/\mu)\,dB(s)\neq 0,\]
since
\[\text{Var}\le(\int_0^t\sin (s/\mu)\,dB(s)\r)=\int_0^t\sin^2(s/\mu)\,ds\to \frac t2,\ \ \ \ \text{as}\ \ \mu\downarrow 0.\]

\medskip

Nevertheless,  problem under consideration can be regularized, in such a way that a counterpart of the Smoluchowski-Kramers approximation is still valid.
To this purpose, there are various ways to regularize the problem. One possible way consists in regularizing the noise (to this purpose, see \cite{cf-2011}  and \cite{jjl} for  the analysis of finite dimensional systems, both in the case of constant and in the case of state dependent magnetic field). Another possible way, which is the one we are using in the present paper, consists  in introducing a small friction proportional to the velocity in equation \eqref{intro-eq}  and considering the regularized problem
\begin{equation} \label{regular-intro}
\left\{
  \begin{array}{l}
    \ds{\mu \frac{\partial^2 u^\e_{\mu}}{\partial t^2}(t) = \Delta u_{\mu}^\e(t) + B(u^\e_\mu(\cdot,t),t)  +\vec{m}\times \frac{\partial u^\e_{\mu}}{\partial t}(t) -\e\,\frac{\partial u^\e_{\mu}}{\partial t}(t)+ G(u^\e_\mu(\cdot,t),t)\frac{\partial w^{Q}}{\partial t}(t),}\\
    \vs
     \ds{u^\e_\mu(0) = u_0, \ \ \ \frac{\partial u^\e_\mu}{\partial t}(0) = v_0, \ \ \ \ \ \ \ u^\e_\mu(\xi,t) = 0, \ \ \xi \in \partial D, }\\
  \end{array} \right.
\end{equation}
which now depends on two small positive parameters $\e$ and $\mu$.
  Our purpose here is showing that, for any fixed $\e>0$, we can take the limit as $\mu$ goes to $0$. Namely, we want to prove that for any $T>0$ and $p\geq1$
\begin{equation}
\label{lim-mu}
\lim_{\mu\to 0}\E\sup_{t \in\,[0,T]}|u^\e_\mu(t)-u_\e(t)|_{L^2(D;\reals^2)}^p=0,\end{equation}
where $u_\e(t)$ is the unique mild solution of the problem
\begin{equation}
\label{limeps}
\left\{
  \begin{array}{l}
    \ds{\frac{\partial u_{\e}}{\partial t}(\xi,t) = \le(J_0+\e\,I\r)^{-1}\Big[ \Delta u_{\e}(\xi,t) + B(u_\e(\cdot,t),t) + G(u_\e(\cdot,t),t) \frac{\partial w^{Q}}{\partial t}(\xi,t) \Big],}\\
    \vs
    \ds{u_\e(\xi,t) = 0, \ \ \xi \in \partial D,\ \ \ \ \ \ u_\e(\xi,0) = u_0(\xi),  \xi \in D, }
  \end{array} \right.
\end{equation}
which is precisely what we get from \eqref{regular-intro} when we put $\mu=0$.

The proof of \eqref{lim-mu} is not at all straightforward. First of all, it requires a thorough analysis of the linear semigroup $S^\e_\mu(t)$ in the space $L^2(D)\times H^{-1}(D)$, associated with the  differential operator
\[A^\e_\mu(u,v)=\frac 1\mu (\mu v,\Delta u-(J_0+\e I)v),\ \ \ \ \ (u,v) \in\,D(A^\e_\mu)=H^1(D)\times L^2(D).\]
Suitable uniform bounds with respect to $\mu$ have to proven in order to prove the convergence in an appropriate sense of the semigroup $S^\e_\mu(t)$ to the semigroup $T_\e(t)$ associated with the linear differential operator  $(J_0+\e I)^{-1}\Delta$ in equation \eqref{limeps}.

Next, as the nonlinearities $B$ and $G$ are assumed to be Lipschitz-continuous, in order to obtain \eqref{lim-mu} the whole point is showing that the stochastic convolution associated with equation \eqref{regular-intro} converges to the stochastic convolution associated with
equation \eqref{limeps}. To this purpose, we have to distinguish the case of additive noise ($G$ constant) and of multiplicative noise ($G$  depending on the state $u$). As a matter of fact, while for additive noise the result is true in any space dimension, for multiplicative noise we are only able to treat the case of space dimension $d=1$ (see also  \cite{cf-2006-mult} for an analogous situation). In both cases, one of the key tools in the proof is the stochastic factorization formula combined with a-priori bounds.

\medskip

Once we have obtained \eqref{lim-mu},  we show that the regularized problems \eqref{regular-intro} and \eqref{limeps} provide a good approximation for the original problems \eqref{intro-eq} and \eqref{intro-eq2}, where the magnetic field is acting in absence of  friction. Thus, we  prove
that for any fixed $\mu>0$ and for any $\e>0$ small enough the solution $u^\e_\mu$ of the regularized system \eqref{regular-intro} is close to the solution of the original system \eqref{intro-eq}.
More precisely,
\begin{equation}
\label{lim-eps-mu}
\lim_{\e\to 0}\E\,\sup_{t \in\,[0,T]}|u^\e_\mu(t)-u_\mu(t)|_{L^2(D;\reals^2)}^p=0,
\end{equation}
and
\begin{equation}
\label{lim-eps-mu-v}
\lim_{\e\to 0}\E\,\sup_{t \in\,[0,T]}\le|\frac{\partial u^\e_\mu}{\partial t}(t)-\frac{\partial u_\mu}{\partial t}(t)\r|_{H^{-1}(D;\reals^2)}^p=0.\end{equation}
In the same way, we prove that
\begin{equation}
\label{lim-eps}
\lim_{\e\to 0}\E\sup_{t \in\,[0,T]}|u_\e(t)-u(t)|_{L^2(D;\reals^2)}^p=0,\end{equation}
where $u(t)$ is the solution of system \eqref{intro-eq2}. To this purpose, we would like to stress that system \eqref{intro-eq2} is not of parabolic type and the semigroup $T_0(t)$ associated with the differential operator
\[J_0^{-1}\Delta(u_1, u_2)=(-\Delta u_2, \Delta u_1)\] is not analytic  in $L^2(D;\reals^2)$ (in fact, it is an isometry). In particular,  equation \eqref{intro-eq2} is not well posed in $L^2(D;\reals^2)$ under the minimal regularity assumptions on the noise required for systems \eqref{intro-eq}, \eqref{regular-intro} and \eqref{limeps} to be well posed and for limit \eqref{lim-mu} to hold. Actually, the noise in system \eqref{intro-eq2} has to be assumed to be taking values in $L^2(D;\reals^2)$ (which means that the covariance of the noise is a trace-class operator).
Moreover, in spite of the fact that both system \eqref{intro-eq} and system \eqref{regular-intro} are well defined under weaker regularity conditions on the noise, limit \eqref{lim-eps-mu} is true only if the covariance is trace-class.

\section{Assumptions and notations}
\label{sec2}

Let us assume that $D$ is a bounded regular domain in $\reals^d$, with $d\geq 1$. In what follows, we shall denote by $H$ the Hilbert space $L^2(D,\reals^2)$, endowed with the scalar product
\[\le<(x_1,y_1),(x_2,y_2)\r>_H=\int_D x_1(\xi)x_2(\xi)\,d\xi+\int_D y_1(\xi)y_2(\xi)\,d\xi,\]
and the corresponding norm $|\cdot|_H$.

Now, let $\hat{A}$ denote the realization of the Laplace operator in $L^2(D)$, endowed with Dirichlet boundary conditions.  Then there exists an orthonormal basis $\{\hat{e}_k\}$ for $L^2(D)$ and a positive sequence $\{\hat{\alpha}_k\}$ such that $\hat{A}\hat{e}_k = -\hat{\alpha}_k \hat{e}_k$, with $0<\hat{\alpha}_1\leq  \hat{\alpha}_k \leq \hat{\alpha}_{k+1}$. Thus, if we define for any $k \in \nat$,
 \[e_{2k-1} = (\hat{e}_k ,0), \ \alpha_{2k} = \hat{\alpha_k},\]
 \[e_{2k} = (0, \hat{e}_k), \ \alpha_{2k+1} = \hat{\alpha_k},\]
 we have that $\{e_k\}_{k=1}^\infty$ is a complete orthonormal basis of $H$. Moreover, if we define
 \[D(A)=D(\hat{A})\times D(\hat{A}),\ \ \ A(x,y)=(\hat{A}x,\hat{A}y),\ \ \ (x,y) \in\,D(A),\]
 we have that
 \[A e_k=-\alpha_k e_k,\ \ \ \ k \in\,\nat.\]

 Next,
for any  $\delta \in \mathbb{R}$, we define $H^\delta$ to be the completion of $C^{\infty}_0(D;\mathbb{R}^2)$ with respect to  the norm
\[|u|_{H^\delta}^2 = \sum_{k=1}^\infty \alpha_k^\delta \left<u, e_k \right>_H^2.\]
Moreover, we define
$\H_\delta := H^\delta \times H^{\delta -1}$, and in the case $\delta=0$ we  simply set   $\H := \H_0$. Finally, for any $(x,y) \in\,\H_\delta$, we  denote
\[\Pi_1(x,y)=x,\ \ \ \ \Pi_2(x,y)=y.\]

The cylindrical Wiener process $w^Q(t)$ is defined as the formal sum
\[w^Q(t) = \sum_{k=1}^\infty Qe_k \beta_k(t),\]
where $Q=(Q_1,Q_2) \in\,{\cal L}(H)$,
 $\{\beta_k\}$ is a sequence of identical, independently distributed one-dimensional, Brownian motions defined on some probability space $(\Omega,\cal{F},\mathbb{P})$ and $\{e_k\}$ is the orthonormal basis of $\H$ introduced above.

Concerning the non-linearity $B$ we assume the following conditions
\begin{Hypothesis} \label{H2}
The mapping  $B:H\times [0,+\infty)\to H$ is measurable. Moreover, for any $T>0$ there exists $\kappa_B(T)>0$ such that
  \[\left|B(x,t) - B(y,t) \right|_H \leq \kappa_B(T)|x-y|_H,\ \ \ \ x,y \in\,H,\ \ \ t \in\,[0,T],\]
  and
  \[\sup_{t \in\,[0,T]}|B(0,t)|_H\leq \kappa_B(T).\]
\end{Hypothesis}
In the case there exists some measurable $b:\reals\times D\times [0,+\infty)\to \reals$ such that for any $x \in\,L^2(D)$ and $t\geq 0$
\[B(x,t)(\xi)=b(x(\xi),\xi,t),\ \ \ \xi \in\,D,\]
then Hypothesis \ref{H2} is satisfied if $b(\cdot,\xi,t):\reals\to\reals$ is Lipschitz continuous and has linear growth, uniformly with respect to $\xi \in\,D$ and $t \in\,[0,T]$, for any $T>0$.

Concerning the diffusion coefficient $G$, we assume the following
\begin{Hypothesis}
\label{H3}
The mapping $G:H\times [0,+\infty)\to{\cal L}(L^\infty(D);H)$ is measurable and for any $T>0$ there exists $\kappa_G(T)>0$ such that
 \[\left|\le[G(x,t) - G(y,t)\r]z \right|_H \leq \kappa_G(T)|x-y|_H|z|_\infty,\ \ \ \ x,y \in\,H,\ \ z \in\,L^\infty(D),\ \ \ t \in\,[0,T],\]
  and
  \[\sup_{t \in\,[0,T]}|G(0,t)z|_H\leq \kappa_G(T)|z|_\infty,\ \ \ \ z \in\,L^\infty(D),\ \ \  t \in\,[0,T].\]
\end{Hypothesis}
In particular, this implies that for any $x,y,z \in\,H$
\begin{equation}
\label{s12}
|[G^\star(x,t)-G^\star(y,t)]z|_{(L^\infty(D))^\prime}\leq \kappa_G(T)|x-y|_H\,|z|_H, \ \ \ \ t \in\,[0,T].
\end{equation}
If for any $x \in\,L^2(D)$ and $z \in\,L^\infty(D)$ we define
\[[G(x,t)z](\xi)=g(x(\xi),\xi,t)z(\xi),\ \ \ \ \ \xi \in\,D,\]
for some measurable
$g:\reals^2 \times D \times [0,+\infty] \to {\cal L}(\reals^2)$, then Hypothesis \ref{H3} is satisfied if
  \[\sup_{\xi \in D}\sup_{t \in [0,T]} | g(x,\xi,t) - g(y,\xi,t)|_{{\cal L}(\reals^2)} \leq \kappa_T |x-y|_{\reals^2}\]
  and that is has linear growth
  \[\sup_{\xi \in D} \sup_{t \in [0,T]} |g(x,\xi,t)|_{{\cal L}(\reals^2)} \leq \kappa_T(1 + |x|_{\reals^2}).\]
Actually, in this case
\[\begin{array}{l}
\ds{|(G(x_1,t) - G(x_2,t))y |_H^2 = \int_D |(g(x_2(\xi),\xi,t) - g(x_2(\xi),\xi,t))y(\xi)|_{\reals^2}^2 d\xi}\\
\vs
\ds{\leq \kappa_T \int_D |x_2(\xi)-x_1(\xi)|_{\reals^2}^2 |y(\xi)|_{\reals^2}^2 d\xi \leq |x_2 -x_1|_H^2 |y|_\infty^2,}
\end{array}
\]
and by the same reasoning
\begin{equation} \label{G-bound}
  |G(x,t)y|_{H} \leq \kappa_T(1+ |x|_H) |y|_\infty.
\end{equation}

\medskip

Now, for any $\mu>0$ and $\d \in\,\reals$, we define on $\H_\d$ the unbounded linear operator
\[A_\mu (u,v)=\frac 1\mu(\mu v,Au-J_0v),\ \ \ \ (u,v) \in\,D(A_\mu)=\H_{\d+1},\]
where
$J_0$ is the skew symmetric $2 \times 2$ matrix
\[J_0 = \begin{pmatrix} 0 & 1 \\ -1 & 0 \end{pmatrix}.\]
It can be proven that $A_\mu$ is the generator of a strongly continuous group of bounded linear operators $\{S_\mu(t)\}_{t\geq 0}$ on each $\H_\d$ (for a proof see \cite[Section 7.4]{pazy}).

Moreover, for any $\mu>0$ we define
\[B_\mu:\H\times [0,+\infty)\to\H,\ \ \ \ \ (z,t) \in\,\H\times [0,+\infty)\mapsto \frac 1\mu(0,B(\Pi_1 z,t)),\]
and
\[G_\mu:\H\times [0,+\infty)\to{\cal L}(L^\infty(D),\H),\ \ \ \ \ (z,t) \in\,\H\times [0,+\infty)\mapsto \frac 1\mu (0,G(\Pi_1 z,t)).\]
With these notations, if we set
\[z_\mu(t)=\le(u_\mu(t),\frac{\partial u_\mu}{\partial t}(t)\r),\]
system \eqref{intro-eq} can be rewritten as the following stochastic equation in the Hilbert space $\H$
\begin{equation}
\label{abstract}
dz_\mu(t)=\le[A_\mu z_\mu(t)+B_\mu(z_\mu(t),t)\r]\,dt+G_\mu(z_\mu(t),t)dw^Q(t),\ \ \ \ \ z_\mu(0)=(u_0,v_0).\end{equation}

\section{The approximating semigroup}
\label{sec3}

In what follows we will consider \eqref{regular-intro}.
For any $\mu,\e>0$ and $\d \in\,\reals$, we define
\[A_{\mu}^{\e}(u,v)=\frac 1\mu(\mu\,v,Au-J_\e v),\ \ \ \ \ (u,v) \in\,D(A_{\mu}^{\e})=\H_{\d+1},\]
where
\[J_\e = J_0 +\e I =  \begin{pmatrix} \e & 1 \\ -1 & \e \end{pmatrix},\ \ \ \ \e>0. \]
As we have seen in the previous section for $A_\mu$, it is possible to prove that for any $\mu, \e>0$ the operator $A_{\mu}^{\e}$ generates a strongly continuous group of bounded linear operators $S_{\mu}^\e(t)$, $t\geq 0$,  on $\H_\d$.

\begin{Lemma}
\label{energy-est-lem}
  For any $(x,y) \in\,\H_\theta$, with $\theta \in\,\reals$, and for any $\mu, \e>0$ let us define
  \[u_{\mu}^\e(t): = \Pi_1 S_\mu^\e(t)(x,y),\ \ \ \ \ \ v_\mu^\e(t):=\Pi_2 S_\mu^\e(t)(x,y).\]
Then
\begin{equation} \label{energy-est-1}
  \mu \left|v_\mu^\e(t) \right|_{H^{\theta-1}}^2 + |u_\mu^\e(t)|_{H^\theta}^2+2\e\int_0^t|v_\mu^\e(s)|_{H^{\theta-1}}^2\,ds
 = \mu|y|_{H^{\theta-1}}^2 + |x|_{H^\theta},
\end{equation}
and
\begin{equation} \label{energy-est-2}
  \mu |u_\mu^\e(t)|_{H^\theta}^2 + \left| \mu v_\mu^\e(t) + J_\e u_\mu^\e(t) \right|_{H^{\theta-1}}^2+2\e\int_0^t|u_\mu^\e(s)|_{H^\theta}^2\,ds
=\mu |x|_{H^\theta}^2 + |\mu y + J_\e x|_{H^{\theta -1}}^2.
\end{equation}
\end{Lemma}

\begin{proof}
Since
\[\frac{\partial u_\mu^\e}{\partial t}(t)=v_\mu^\e(t),\]
we have
\begin{equation}
\label{s1}
\mu \frac{\partial v_\mu^\e}{\partial t} (t) + J_\e v_\mu^\e (t) = A u_\mu^\e(t).\end{equation}
Then, if we take the scalar product of both sides above with $v_\mu^\e$ in $H^{\theta-1}$, we get
\[ \frac 12 \frac d{dt}\le(\mu\,\left| v_\mu^\e(t) \right|_{H^{\theta-1}}^2 + \left| u_\mu^\e(t) \right|_{H^\theta}^2\right)=-\e \left|v_\mu^\e(t) \right|_{H^{\theta-1}}^2,\]
which implies \eqref{energy-est-1}, as
$u_\mu^\e(0) =x$ and $v_\mu^\e(0) =  y$.

  Next, by using again \eqref{s1}, we have
  \[
  \begin{array}{l}
  \ds{\frac{d}{dt} \frac{1}{2} \left| \mu v_\mu^\e(t) + J_\e u_\mu^\e(t)  \right|_{H^{\theta-1}}^2 = \left< \mu v_\mu^\e(t) + J_\e u_\mu^\e(t), A u_\mu^\e(t)  \right>_{H^{\theta -1} }} \\
  \vs
  \ds{ = -\frac{\mu}{2} \frac{d}{dt} \left| u_\mu^\e(t) \right|_{H^\theta}^2 - \e \left| u_\mu^\e(t) \right|_{H^\theta}^2, }
  \end{array}
  \]
  and then, integrating with respect to time, we get \eqref{energy-est-2}.
\end{proof}

Notice that in particular this implies that for any $\mu,\e>0$ there exists $c_{\mu,\e}>0$ such that for any $(x,y) \in\,\H_\theta$
\[\int_0^\infty |S^\e_\mu(t)(x,y)|_{\H_\theta}^2\,dt\leq \frac {c_{\mu,\e}}{2\e}|(x,y)|_{\H_\theta}^2.\]
As a consequence of the Datko theorem (see \cite{ben} for a proof), we can conclude that there exist $M_{\mu,\e},$ and  $\omega_{\mu,\e}>0$ such that
\[\|S_\mu^\e(t)\|_{{\cal L}(\H_\theta)}\leq M_{\mu,\e}\,e^{-\omega_{\mu,\e}t},\ \ \ \ t\geq 0.\]

\begin{Lemma} \label{Smu-bound-lem}
  For any $\mu, \e>0$,  and for any $\theta \in \reals$ and $\gamma \in [0,1]$ it holds
\begin{equation}
\label{s2}
\left|\Pi_1 S_\mu^\e(t)(0,y) \right|_{H^\theta} \leq 2^{\gamma} \mu^{\frac{1+\gamma}{2}}  |y|_{H^{\theta+\gamma-1}},\ \ \ t\geq 0,\ \ \ \ y \in\,H^{\theta+\gamma-1}.\end{equation}
\end{Lemma}
\begin{proof}
  Let $u_\mu^\e(t):= \Pi_1 S_\mu^\e(t)(0,y)$.  By \eqref{energy-est-1},
  \[\left|u_\mu^\e(t) \right|_{H^{\theta+\gamma}}^2 \leq \mu |y|_{H^{\theta-1}}^2.\]
  By \eqref{energy-est-2} and \eqref{energy-est-1},
  \[\begin{array}{l}
  \ds{\left|u_\mu^\e(t)\right|_{H^{\theta + \gamma -1}}^2 \leq (1+ \e^2) \left|u_\mu^\e(t) \right|_{H^{\theta + \gamma -1}}^2 = \left| J_\e u_\mu^\e(t) \right|_{H^{\theta + \gamma-1}}^2 }\\
  \vs
  \ds{\leq 2 \left| \mu \frac{\partial u_\mu^\e}{\partial t}(t) + J_\e u_\mu^\e(t) \right|_{H^{\theta + \gamma-1}}^2 +2 \mu^2 \left| \frac{\partial u_\mu^\e}{\partial t}(t)  \right|_{H^{\theta + \gamma -1}}^2\leq 4 \mu^2  |y|_{H^{\theta + \gamma - 1}}^2. }\end{array}  \]
Then, since for any $x \in H^{\theta+ \gamma}$
  \[|x|_{H^\theta} \leq |x|_{H^{\theta+ \gamma-1}}^{\gamma}|x|_{H^{\theta+ \gamma}}^{1-\gamma},\]
we conclude that
  \[|u_\mu^\e(t)|_{H^\theta} \leq |u_\mu^\e(t)|_{H^{\theta+\gamma-1}}^\gamma |u_\mu^\e(t)|_{H^{\theta+ \gamma}}^{1-\gamma} \leq 2^{\gamma} \mu^{\frac{1+\gamma}{2}} |y|_{H^{\theta+\gamma-1}}.\]
\end{proof}

Now, for any $\mu>0$ we define the bounded linear operator
\[Q_\mu:H\to \H,\ \ \ \ x \in\,H\mapsto \frac1\mu(0,Qx) \in\,\H.\]

\begin{Lemma}
\label{l1}
Assume that there exists a non-negative sequence $\{\la_k\}_{k \in\,\nat}$ such that
\[Qe_k=\la_k e_k,\ \ \ \ k \in\,\nat.\]
Then,
for any $0<\delta<1$ and $\e>0$ there exists a constant $c= c(\e, \d)>0$ such that for any $k \in \nat$ and $\theta>0$
  \begin{equation} \label{bound-1}
    \sup_{\mu>0} \int_0^\infty s^{-\d} \left| \Pi_1 S_\mu^\e(s) Q_\mu e_k \right|_H^2 ds \leq c\, \frac{\la_k^2}{\alpha_k^{1-\d}},
  \end{equation}
  and
  \begin{equation} \label{bound-3}
    \sup_{\mu>0} \mu^{1+ \d} \int_0^\infty s^{-\d} \left| \Pi_2 S_\mu^\e(s) Q_\mu e_k \right|_{H^{\theta-1}}^2 ds \leq c\, \frac{\la_k^2}{\alpha_k^{1-\theta}}.
  \end{equation}
\end{Lemma}

\begin{proof}
We have
\[\int_0^\infty s^{-\d} \left| \Pi_1 S_\mu^\e(s) Q_\mu e_k \right|_H^2 ds=\int_0^{\a_k^{-1}} s^{-\d} \left| \Pi_1 S_\mu^\e(s) Q_\mu e_k \right|_H^2 ds+
\int_{\a_k^{-1}}^\infty s^{-\d} \left| \Pi_1 S_\mu^\e(s) Q_\mu e_k \right|_H^2 ds.\]
Due to \eqref{s2}, with $\theta=0$ and $\gamma=1$, we have
\begin{equation}
\label{s3}
\int_0^{\a_k^{-1}} s^{-\d} \left| \Pi_1 S_\mu^\e(s) Q_\mu e_k \right|_H^2 ds\leq 4\,|Q e_k|_H^2\int_0^{\a_k^{-1}} s^{-\d}\,ds=\frac{4}{1-\d}\frac {\la_k^2}{\a_k^{1-\d}}.\end{equation}
Moreover, due to \eqref{energy-est-2} we have
\[\int_{\a_k^{-1}}^\infty s^{-\d} \left| \Pi_1 S_\mu^\e(s) Q_\mu e_k \right|_H^2 ds\leq \frac1{2\e} \a_k^{-\d}|Q e_k|_{H^{-1}}^2=\frac1{2\e}\frac{\la_k^2}{\a_k^{1-\delta}}.\]
Together with \eqref{s3} this implies \eqref{bound-1}.

  To establish \eqref{bound-3}, we write
 \[\int_0^\infty s^{-\d} \left| \Pi_2 S_\mu^\e(s) Q_\mu e_k \right|_{H^{\theta-1}}^2 ds=\int_0^\mu s^{-\d} \left| \Pi_2 S_\mu^\e(s) Q_\mu e_k \right|_{H^{\theta-1}}^2 ds+\int_\mu^\infty s^{-\d} \left| \Pi_2 S_\mu^\e(s) Q_\mu e_k \right|_{H^{\theta-1}}^2 ds.\]
Thanks to \eqref{energy-est-1} we have
  \[\int_0^\mu s^{-\d} \left| \Pi_2 S_\mu^\e(s) Q_\mu e_k \right|_{H^{\theta-1}}^2 ds\leq \frac 2{\mu^2}|Qe_k|_{H^{\theta-1}}^2\int_0^\mu s^{-\d}\,ds= \frac1{\mu^{1+\d}}\frac{2}{1-\d}\frac{\lambda_k^2}{\alpha_k^{1-\theta}}, \]
  and
  \[\int_\mu^\infty s^{-\d} \left| \Pi_2 S_\mu^\e(s) Q_\mu e_k \right|_{H^{\theta-1}}^2 ds\leq \frac 1{2\e}\mu^{-\d}\frac1{\mu}|Qe_k|_{H^{\theta-1}}^2=
\frac 1{2\e}  \frac1{\mu^{1+\d}}\frac{\lambda_k^2}{\alpha_k^{1-\theta}},\]
and these two estimates together imply \eqref{bound-3}.
\end{proof}

Now, for any $\e>0$ we define
\[A_\e:=J_\e^{-1} A=\frac{1}{1+\e^2}\begin{pmatrix} \e  & -1 \\ 1 & \e \end{pmatrix}\hat{A},\]
and we denote by $T_\e(t)$, $t\geq 0$, the strongly continuous semigroup generated by $A_\e$ in $H^\theta$, for any $\theta \in\,\reals$.
Moreover, we denote
\[Q_\e=J_\e^{-1}Q.\]

\begin{Lemma}
\label{l7}
We have
  \begin{equation} \label{Se-bound}
    \|T_\e(t) \|_{{\cal L}(H^\theta)}  \leq e^{-\frac{\e\a_1}{1+\e^2}t},\ \ \ \ \ t\geq 0.
  \end{equation}

 Moreover, if there exists a non-negative sequence $\{\la_k\}_{k \in\,\nat}$ such that
\[Qe_k=\la_k e_k,\ \ \ \ k \in\,\nat,\]
then, for any $0<\d<1$ and $\e>0$ there exists a constant $c=c(\d,\e)$ such that for any $k \in\,\nat$
  \begin{equation} \label{bound-2}
    \int_0^\infty s^{-\d} \left| T_\e(s) Q_\e e_k \right|_H^2 ds \leq c \frac{\lambda_k^2}{\alpha_k^{1-\d}}.
  \end{equation}
  Finally, for any  $k \in\,\nat$
  \begin{equation}
  \label{s45}
    \int_0^T s^{-\d} \left| T_\e(s) Q_\e e_k \right|_H^2 ds \leq \frac{1}{1-\d}\,T^{1-\d}\,\lambda_k^2.
    \end{equation}

\end{Lemma}
\begin{proof}
  Let $x \in H^\theta$, and $u_\e(t) = T_\e(t)x.$ This means
\[\frac{\partial u_\e}{\partial t}(t) =  A_\e u_\e(t) = \frac{1}{1+ \e^2}
\begin{pmatrix}
  \e & -1 \\ 1 & \e
\end{pmatrix}\hat{A} u_\e(t)\]
  Then, if we take the scalar product in $H^\theta$ of the above equation by $u_{\e}(t)$
  \[\begin{array}{l}
  \ds{\frac d{dt}\le(|u_{\e}(t)|_{H^\theta}^2\r)=-\frac {2\e}{1+\e^2}\le(|u_{\e}(t)|_{H^{1+\theta}}^2\r)}\\
  \vs
\ds{\leq -\frac {2\e\a_1}{1+\e^2}\le(|u_{\e}(t)|_{H^{\theta}}^2\r),}
\end{array}
\]
and this implies \eqref{Se-bound}.

In  order to prove \eqref{bound-2}, we observe that if $u_\e(t) = T_\e(t) Q_\e e_k$, then
\[\frac{\partial u_\e}{\partial t}(t) = -\alpha_k J_\e^{-1} u_\e(t).\]
By the same arguments that we used above,
\[|u_\e(t)|_H = e^{-\frac{\e \a_k}{\e^2+1}t} |Q_\e e_k|_H = \frac{\lambda_k}{\sqrt{\e^2+1}} e^{-\frac{\e \a_k}{\e^2+1}t},\]
and therefore,
\[ \begin{array}{l}
\ds{   \int_0^\infty s^{-\d} \left| T_\e(s) Q_\e e_k \right|_H^2 ds \leq \frac {\la_k^2}{(1+\e^2)}\int_0^\infty s^{-\d} e^{-\frac {2\e\a_k}{1+\e^2} s} ds}\\
\vs
\ds{  \leq \frac{\la_k^2}{\e^2+1}\le(\int_0^{\a_k^{-1}} s^{-\d} ds+  \a_k^\d \int_{\a_k^{-1}}^\infty e^{-\frac {2\e\a_1}{1+\e^2} s} ds \r) \leq \frac{\la_k^2}{\e^2+1}\le(\frac 1{1-\d}\a_k^{\d-1}+\a_k^\d\frac {1+\e^2}{2\e\a_k}\r)\leq c(\a,\e)\, \frac{\lambda_k^2}{\alpha_k^{1-\d}}.}
\end{array}\]

Finally, in order to prove \eqref{s45}, we notice that
\[ \int_0^T s^{-\d} \left| T_\e(s) Q_\e e_k \right|_H^2 ds\leq \la_k^2\int_0^T s^{-\d}\,ds=\frac{T^{1-\d}}{1-\d}\,\la_k^2.\]

\end{proof}

In view of the previous estimates for $S_\mu^\e(t)$ and $T_\e(t)$, we can prove the following convergence result.
\begin{Theorem}
  For any $\e >0$, $0<t_0<T$, and   $n \in \nat$,
  \begin{equation} \label{Smu-to-Se-1}
    \lim_{\mu \to 0} \sup_{t \leq T} \sup_{|x|_H\leq 1} \left| \Pi_1 S_\mu^\e(t)(P_n x,0) - T_\e(t) P_n x \right|_{H} =0,
  \end{equation}
  and
  \begin{equation} \label{Smu-to-Se-3}
    \lim_{\mu \to 0} \sup_{0<t_0 \leq t\leq T} \sup_{|y|_H \leq 1} \left|\frac{1}{\mu} \Pi_1 S_\mu^\e(t)(0,P_n y) - T_\e(t)J_\e^{-1} P_n y \right|_H = 0,
  \end{equation}
  where $P_n$ is the projection of $H$ onto the $n$-dimensional subspace $H_n:=\text{\em{span}}\{e_1,\ldots,e_{2n}\}$.
\end{Theorem}

\begin{proof}
  Fix $k \in\,\nat$, and let us consider the function $u_\mu^\e(t) = \Pi_1 S_\mu^\e(t)\left(x,\frac{y}{\mu}\right)$, with $x,y \in \span\{e_{2k-1}, e_{2k}\}$.
 We have
  \[
  \left\{
  \begin{array}{l}
  \ds{\mu \frac{\partial^2 u_\mu^\e}{\partial t^2}(t) + J_\e \frac{\partial u^\e_\mu}{\partial t}(t) = - \alpha_{2k} u^\e_\mu(t)}\\
  \vs
  \ds{u^\e_\mu(0) = x, \ \ \ \ \frac{\partial u^\e_\mu}{\partial t}(0) = \frac{y}{\mu},}\end{array}\right.\]
so that
  \[\frac{d}{dt} \left(e^{\frac{J_\e}{\mu}t}\frac{\partial u^\e_\mu}{\partial t}(t) \right) = -\frac{\alpha_{2k}}{\mu}\, e^{\frac{J_\e}\mu t}u^\e_\mu(t).\]
  By integrating in time, we see that
  \[\frac{\partial u^\e_\mu}{\partial t}(t) =  e^{-\frac{J_\e}{\mu}t} \frac{y}{\mu} -\frac{\alpha_{2k}}{\mu} \int_0^t e^{-\frac{J_\e}{\mu}(t-s)} u^\e_\mu(s) ds.  \]
  Integrating once again, and exchanging the order of integration, we conclude that $\Pi_1 S^\e_\mu(t)(x,y/\mu)=u^\e_\mu(t)$ solves
  \begin{equation} \label{int-by-parts-eq}
   u^\e_\mu(t) = x +  \left(I - e^{-\frac{J_\e}{\mu}t}\right) J_\e^{-1} y -\alpha_{2k} \int_0^t \left(I-e^{-\frac{J_\e}{\mu}(t-s)}\right)J_\e^{-1}u^\e_\mu(s) ds.
  \end{equation}

Now, since $T_\e(t)x$ solves the equation
  \[T_\e(t)x = x -\alpha_{2k} \int_0^t J_\e^{-1} T_\e(s)x ds,\]
 from \eqref{int-by-parts-eq} and \eqref{energy-est-1}, we get
  \[\begin{array}{l}
  \ds{\left|\Pi_1S^\e_\mu(t)(x,0) - T_\e(t)x \right|_H }\\
  \vs
  \ds{\leq \alpha_{2k} \int_0^t \left|J_\e^{-1}( \Pi_1S^\e_\mu(s)(x,0)- T_\e(s)x) \right|_H ds + \alpha_{2k} \int_0^t |e^{-\frac{J_\e}{\mu}(t-s)} J_\e^{-1}S^\e_\mu(s)(x,0)|_H ds.}\\
  \vs
  \ds{\leq \alpha_{2k} \int_0^t \left| \Pi_1S^\e_\mu(s)(x,0)- T_\e(s)x \right|_H ds+ \alpha_{2k} \int_0^t e^{-\frac \e\mu(t-s)}\,ds|x|_H.}
  \end{array}\]
From Gr\"onwall's inequality we see that for $0 \leq t \leq T$,
  \[\left|\Pi_1S^\e_\mu(t)(x,0) - T_\e(t)x \right|_H \leq \frac{\mu}{\e}\, \a_{2k}|x|_H e^{\a_{2k} T}, \]
 and  this yields \eqref{Smu-to-Se-1}.

  We prove \eqref{Smu-to-Se-3} analogously, by taking $x=0$ in \eqref{int-by-parts-eq}. In this case, thanks to \eqref{s2} and the fact that $\|J_\e^{-1}\|_{\mathcal{L}(\reals^2)} \leq 1.$
  \[\begin{array}{l}
  \ds{\left|\frac 1\mu \Pi_1 S_\mu^\e(t)(0,y) - T_\e(t) J_\e^{-1} y \right|_H \leq \left|e^{-\frac{J_\e}{\mu}t} J_\e^{-1} y \right|_H}\\
 \vs
 \ds{ + \alpha_{2k} \int_0^t \left|J_\e^{-1} \le(\frac 1\mu \Pi_1 S^\e_\mu(s)(0,y)- T_\e(s)J_\e^{-1} y\r) \right|_H ds + \alpha_{2k} \int_0^t \left|e^{-\frac{J_\e}{\mu}(t-s)} J_\e^{-1} \frac 1\mu \Pi_1 S^\e_\mu(s)(0,y) \right|_H ds}\\
 \vs
 \ds{\leq e^{-\frac \e\mu t}|y|_H+ \alpha_{2k} \int_0^t \left|\frac 1\mu \Pi_1 S^\e_\mu(s)(0,y)- T_\e(s)J_\e^{-1} y \right|_H ds+\alpha_{2k} \int_0^t e^{-\frac \e\mu (t-s)}\,ds|y|_H}\\
 \vs
 \ds{\leq \le(e^{-\frac \e\mu t}+\frac {\a_{2k}\mu}\e\r)|y|_H+\alpha_{2k} \int_0^t \left|\frac 1\mu \Pi_1 S^\e_\mu(s)(0,y)- T_\e(s)J_\e^{-1} y \right|_H ds.}
 \end{array} \]
  By Gr\"onwall's inequality,
  \[\left| \frac 1\mu \Pi_1 S_\mu^\e(t)(0,y) - T_\e(t) J_\e^{-1} y \right|_H \leq \left( e^{-\frac{\e t}{\mu}}\left| y \right|_H + \frac{\mu \a_{2k}}{\e}  |y|_H \right)e^{\a_{2k}T},\ \ \ \ \ t \in\,[0,T],\]
  and this implies \eqref{Smu-to-Se-3}.
\end{proof}

\begin{Corollary}
  For any $\e>0$ and $T>0$ and for any $(x,y) \in \H$,
  \begin{equation} \label{Smu-to-Se-at-a-point}
    \lim_{\mu \to 0} \sup_{t \leq T} |\Pi_1 S_\mu^\e(t)(x,y) - T_\e(t) x|_H=0.
  \end{equation}
Moreover, for any $y \in H$ and $0< t_0 \leq T$,
  \begin{equation} \label{Smu-to-Se-3-at-a-point}
    \lim_{\mu \to 0} \sup_{t_0 \leq t \leq T} \left|\frac 1\mu\Pi_1 S_\mu^\e(t) \left(0, y \right) - T_\e(t) J_\e^{-1} y \right|_H =0.
  \end{equation}
\end{Corollary}
\begin{proof}
 We have
   \[ \begin{array}{l}
  \ds{ |\Pi_1 S_\mu^\e(t)(x,y) - T_\e(t) x|_H
  \leq |\Pi_1 S_\mu^\e(t)(0,y) |_H+  |\Pi_1 S_\mu^\e(t) (P_nx,0) - T_\e(t) P_nx|_H}\\
  \vs
  \ds{
  +  |\Pi_1S_\mu^\e(t) (x-P_n x,0)| + \sup_{t \leq T} | T_\e(t) (x-P_nx)|_H:= I_1(t) +\sum_{j=2}^4 I_{n,j}(t).}
   \end{array}\]
  By \eqref{energy-est-1} and \eqref{Se-bound}, for any $\eta>0$ there exists $n_\eta \in\,\nat$ such that
  \[I_{n_\eta,3}(t)+I_{n_\eta,4}(t) \leq 3 |x-P_{n_\eta}x|_H\leq \frac{\eta}3,\ \ \ \ t\geq 0.\]
 Moreover, by \eqref{Smu-to-Se-1}, we can then find $\mu_1$ such that for $\mu<\mu_1$
  \[\sup_{t \in\,[0,T]}\,I_{n_\eta,2}(t) < \frac{\eta}{3},\]
and then since from \eqref{energy-est-1}
  \[\sup_{t\geq 0}I_1(t) \leq 2 \mu |y|_H,\]
we can conclude that
\[\sup_{t \in\,[0,T]}|\Pi_1 S_\mu^\e(t)(x,y) - T_\e(t) x|_H\leq \eta,\ \ \ \mu\leq \mu_0,\]
and
 \eqref{Smu-to-Se-at-a-point} follows from the arbitrariness of $\eta>0$.

In order to prove \eqref{Smu-to-Se-3-at-a-point}, we have
  \[\begin{array}{l}
  \ds{\left|\frac 1\mu S_\mu^\e(t)\left(0, y \right) - T_\e(t) J_\e^{-1} y \right|_H \leq  \left| \frac 1\mu S_\mu^\e(t) \left(0, P_n y \right) - T_\e(t) J_\e^{-1} P_n y \right|_H}\\
  \vs
  \ds{ +  \left|\frac 1\mu \Pi_1S_\mu^\e(t) \left( 0 , y-P_n y \right) \right|_H + \left| T_\e(t) J_\e^{-1} (y-P_n y) \right|_H : = \sum_{j=1}^3I_{n,j}(t) . }
  \end{array}\]
  By Lemma \ref{Smu-bound-lem} and \eqref{Se-bound}, we have
  \[I_{n,2}(t)+I_{n,3}(t) \leq c | y - P_n y|_H,\ \ \ \ \ t \geq 0.\]
    Then, for any $\eta>0$ we can fix $n_\eta \in\,\nat$ such that
    \[\sup_{t\geq 0}I_{n,2}(t)+I_{n,3}(t)\leq \frac \eta 2.\]
Moreover, thanks to
\eqref{Smu-to-Se-3}, we can find $\mu_0$ such that for all $\mu< \mu_0$,
  \[\sup_{t \in\,[t_0,T]}I_{n_\eta,1}(t) < \frac{\eta}{2}.\]
  Because $\eta>0$ was arbitrary, \eqref{Smu-to-Se-3-at-a-point} follows.
\end{proof}

\begin{Corollary}
\label{cor3.7}
  For any $\e>0$, $T>0$ and $p\geq 1$ and for any $\psi \in L^p(\Omega; L^p([0,T];H))$,
  \begin{equation} \label{Smu-to-Se-2}
    \lim_{\mu \to 0} \,\E \sup_{t \in\,[0,T]} \left|\frac{1}{\mu} \int_0^t \Pi_1 S_\mu^\e(t-s) (0,\psi(s)) ds - \int_0^t T_\e(t-s) J_\e^{-1} \psi(s) ds \right|_H^p = 0.
  \end{equation}
\end{Corollary}
\begin{proof}
For any $\psi \in L^p(\Omega;L^p([0,T];H))$ and $n \in\,\nat$, let us define
  \[\psi_n(t) = I_{\{|\psi(t)|_H\leq n\}}P_n\psi(t),\ \ \ \ t \in\,[0,T]. \]
  We have clearly $\psi_n \in\,L^\infty(\Omega\times [0,T];H_n)$ and by the dominated convergence theorem
  \[\lim_{n\to\infty}\E|\psi_n-\psi|^p_{L^p(0,T;H)}=0.\]
If for any $\mu,\e>0$ and $t\geq 0$ we define
  \[\Phi^\mu_\e(t) y =  \frac{1}{\mu} \Pi_1 S_\mu^\e(t) (0, y) - T_\e(t) J_\e^{-1} y,\ \ \ \ y \in\,H, \]
for any $0<\d<t$, we have
  \begin{equation*}
   \begin{array}{l}
     \ds{ \frac{1}{\mu} \int_0^t \Pi_1 S_\mu^\e(t-s) (0,\psi(s)) ds - \int_0^t T_\e(t-s) J_\e^{-1} \psi(s) ds}\\
     \vs
     \ds{  =   \int_0^t \Phi^\mu_\e(t-s) (\psi(s) - \psi_n(s) ) ds + \int_0^{t-\delta}  \Phi^\mu_\e(t-s) \psi_n(s) ds}\\
     \vs
     \ds{+
       \int_{t-\delta}^t  \Phi^\mu_\e(t-s) \psi_n(s)  ds :=  I_{n,1}(t)+I_{n,2}(\d,t)+ I_{n,3}(\d,t).}
   \end{array}
  \end{equation*}
As a consequence of Lemma \ref{Smu-bound-lem} and \eqref{Se-bound}, we have that for any fixed $\e>0$
\[\sup_{t  \in\,[0,T]}\, \sup_{\mu>0}\,\|\Phi^\mu_\e(t)\|_{{\cal L}(H)} := M  < +\infty,\]
so that
\[\E\sup_{t \in\,[0,T]}|I_{n,1}(t)|_H^p\leq T^{p-1} M^p \E|\psi - \psi_n|_{L^p([0,T];H)}^p.\]
Therefore, for any $\eta>0$ there exists $n_\eta \in\,\nat$ such that
\begin{equation}
\label{s4}
\E\sup_{t \in\,[0,T]}|I_{n_\eta,1}(t)|_H^p<\frac{\eta}3.
\end{equation}
Next, since $|\psi_{n_\eta}(s)|_H\leq n_\eta$, for every $s \in\,[0,T]$, we have
\begin{equation}
\label{s5}
\d_{\eta}(t)= \le(\frac{\eta}{3}\r)^{\frac 1p}\frac 1{n_\eta M}\wedge t\Longrightarrow \E \sup_{t \in\,[0,T]}|I_{n,3}(\d_{\eta}(t),t)|_H^p\leq \frac\eta 3.
\end{equation}
Finally, by \eqref{Smu-to-Se-3} we can find $\mu_0>0$ small enough so that for $\mu<\mu_0$
\[\sup_{t \in\,[\d_{\eta}(t),T]}\,\sup_{|y|_H\leq n_\eta} |\Phi^\mu_\e(t) P_{n_\eta}y|_{H}<\le(\frac{\eta}{3T}\r)^{\frac 1p},\]
so that
\[\E \sup_{t \in\,[0,T]}|I_{n,2}(\d_{\eta}(t),t)|_H^p\leq \frac\eta 3.\]
Together with \eqref{s4} and \eqref{s5}, this implies \eqref{Smu-to-Se-2}.

\end{proof}

\section{Approximation by small friction for additive noise } \label{sec4}
In this section, we assume that the noisy perturbation in system \eqref{intro-eq} is of additive type, that is $G(x,t) = I$, for any $x \in\,H$ and $t\geq 0$. Moreover, we assume that the covariance operator $Q$ satisfies the following condition.
\begin{Hypothesis}
\label{H5}
There exists a non-negative sequence $\{\la_k\}_{k \in\,\nat}$ such that $Q e_k=\la_k e_k$, for any $k \in\,\nat$. Moreover, there exists $\d>0$ such that
\[\sum_{k=1}^\infty \frac{\la_k^2}{\a_k^{1-\d}}<\infty.\]
\end{Hypothesis}

With the notations we have introduced in Sections \ref{sec2} and \ref{sec3}, if we denote
\[z_\mu^\e(t)=(u_\mu^\e(t),\frac{\partial u_\mu^\e}{\partial t}(t)),\ \ \ t \geq 0,\]
the regularized system \eqref{regular-intro} can be rewritten as the abstract evolution equation
\begin{equation}
\label{reg-abst}
dz_{\mu}^\e(t)=\le[A^\e_\mu z^\e_\mu(t)+B_\mu(z_\mu^\e(t),t)\r]\,dt+Q_\mu dw(t),\ \ \ \ \ z_\mu^\e(0)=(x,y)
\end{equation}
in the Hilbert space $\H$.

Our purpose here is to prove that for any fixed $\e>0$ the process $u_\mu^\e(t)$ converges to the solution $u_\e(t)$ of the following system of stochastic PDEs
\begin{equation} \label{first-order-eq-w-frict}
\left\{
  \begin{array}{l}
    \ds{\frac{\partial u_\e}{\partial t}(t) = J_\e^{-1} \Delta u_\e(t) + B_\e(u_\e(t),t) + \frac{\partial w^{Q_\e}}{\partial t}}\\
    \vs
    \ds{u_\e(0) = u_0,\ \ \ \ u_\e(\xi,t)=0,\ \ \ \xi \in\,\partial D,}
  \end{array}
  \right.
\end{equation}
where for any $\e>0$ we have defined $Q_\e=J_\e^{-1} Q$ and
\[B_\e(x,t)=J_\e^{-1} B(x,t),\ \ \ x \in\,H,\ \ \ t\geq 0.\]
Notice that with these notations, system \eqref{first-order-eq-w-frict} can be rewritten as the abstract evolution equation
\begin{equation}
\label{first-abst}
du_\e(t)=\le[A_\e u_\e(t)+B_\e(u_\e(t),t)\r]\,dt+Q_\e dw(t),\ \ \ \ u_\e(0)=u_0,\end{equation}
in the Hilbert space $H$.

\medskip

According to Lemma \ref{l1}, due to Hypothesis \ref{H5} for any $t\geq 0$ we have
\[\int_0^t s^{-\d}\sum_{k=1}^\infty |S^\e_\mu(t-s)Q_\mu e_k|^2_{\H}\,ds\leq c\le(1+\mu^{-(1+\d)}\r)\sum_{k=1}^\infty\frac{\la_k^2}{\a_k^{1-\d}}.\]
This implies that the stochastic convolution
\[\Gamma^\e_\mu(t):=\int_0^t S^\e_\mu(s) Q_\mu\,dw(s),\ \ \ \ t\geq 0,\] takes values in $L^p(\Omega;C([0,T];\H))$, for any $T>0$ and $p\geq 1$ (for a proof see \cite{dpz}).
Therefore, as the mapping $B_\mu(\cdot,t):\H\to\H$ is Lipschitz-continuous, uniformly with respect to $t \in\,[0,T]$, we have that there exists a unique process $z^\e_\mu \in\,L^p(\Omega;C([0,T];\H))$  which solves equation \eqref{reg-abst} in the mild sense, that is
\[z^\e_\mu(t)=S^\e_\mu(t)(u_0,v_0)+\int_0^t S^\e_\mu(t-s)B_\mu(z^\e_\mu(s),s)\,ds+\Gamma^\e_\mu(t).\]

In the same way, due to \eqref{bound-2} we have that the stochastic convolution
\[\Gamma_\e(t):=\int_0^t T_\e(s) Q_\e\,dw(s),\ \ \ \ t\geq 0,\] takes values in $L^p(\Omega;C([0,T];H))$, for any $T>0$ and $p\geq 1$, so that, as the mapping $B_\e(\cdot,t):H\to H$ is Lipschitz-continuous, uniformly with respect to $t \in\,[0,T]$, we can conclude that there exists a unique process $u_\e \in\,L^p(\Omega;C([0,T];H))$ solving equation \eqref{first-abst} in mild sense, that is
\[
  u_\e(t) = T_\e(t)u_0 + \int_0^t T_\e(t-s) B_\e(u_\e(s),s) ds + \Gamma_\e(t).
\]

\begin{Theorem} \label{conv-in-mu-thm}
Under Hypotheses \ref{H2} and \ref{H5}, for any $\e>0$, $T>0$ and $p\geq 1$ and for any initial conditions $z_0=(u_0,v_0) \in\,\H$, we have
  \begin{equation}
    \lim_{\mu \to 0} \E \sup_{t\leq T} \left|u_\mu^\e(t) - u_\e(t) \right|_H^p  = 0.
  \end{equation}
\end{Theorem}

\begin{proof}
Due to Lemma \ref{Smu-bound-lem} and the Lipschitz continuity of $B$, we have
  \begin{equation*}
    \begin{array}{l}
      \ds{\left|u_\mu^\e(t) - u_\e(t) \right|_{H} \leq \left| \Pi_1 S_\mu^\e(t)z_0 - T_\e(t)u_0 \right|_{H} +c \int_0^t \left|u_\mu^\e(s) - u_\e(s) \right|_{H}ds} \\
      \vs
      \ds{+ \left|\int_0^t \left[\Pi_1 S_\mu^\e(t-s)  B_\mu((u_\e(s),0),s) - T_\e(t-s) B_\e(u_\e(s),s)  \right] ds \right|_{H}+|\Gamma_\mu^\e(t)-\Gamma_\e(t)|_H,}    \end{array}
  \end{equation*}
and then, from the Gr\"onwall's Lemma, for any $p\geq 1$ we get
  \begin{equation*}
    \begin{array}{l}
      \ds{\sup_{t \in\,[0,T]}\left|u_\mu^\e(t) - u_\e(t) \right|^p_{H} \leq c_{p}(T) \sup_{t \in\,[0,T]}\left| \Pi_1 S_\mu^\e(t)(u_0,v_0) - T_\e u_0 \right|^p_{H}} \\
      \vs
      \ds{+c_{p}(T)\sup_{t \in\,[0,T]}\left|\int_0^t \left( \Pi_1 S_\mu^\e(t-s)B_\mu(X_\e(s),s) - T_\e(t-s)B_\e(X_\e(s),s) \right) ds \right|^p_{H}}\\
      \vs
      \ds{+ c_{p}(T)\sup_{t \in\,[0,T]} \left|\Gamma_\mu^\e(t)-\Gamma_\e(t)\right|^p_{H}:= c_p(T)\,\sum_{k=1}^3 \sup_{t \in\,[0,T]}|I_k(t)|_H^p.}
    \end{array}
  \end{equation*}

By \eqref{Smu-to-Se-at-a-point}
\[\lim_{\mu\to 0}\sup_{t \in\,[0,T]}|I_{1}(t)|_H=0.\]
Moreover, by \eqref{Smu-to-Se-2}, we know that
\[\lim_{\mu\to 0} \E\sup_{t \in\,[0,T]}| I_2(t)|_H^p=0.\]

The analysis of $I_3(t)$ is more delicate. By using the factorization method (see \cite[Chapter 5]{dpz}) for any $\alpha \in (0, 1)$ we have
 \[\begin{array}{l}
 \ds{\frac{\pi}{\sin(\pi \alpha)}\,I_3(t) =    \int_0^t (t-\sigma)^{\alpha -1} \int_0^\sigma (\sigma - s)^{-\alpha}  \left( \Pi_1 S_\mu^\e(t-s)Q_\mu - T_\e(t-s) Q_\e \right) dw(s) d\sigma}\\
 \vs
 \ds{=\int_0^t (t-\sigma)^{\alpha -1} T_\e(t-\sigma)  Y^\a_{\mu,1}(\si)d\sigma+\int_0^t (t-\sigma)^{\alpha-1} \Pi_1 S_\mu^\e(t-\sigma) (0,Y^\a_{\mu,2}(\si)) d\sigma}\\
 \vs
 \ds{+
 \int_0^t (t-\sigma)^{\alpha - 1} \left[ \Pi_1 S_\mu^\e(t - \sigma)\Pi_1^\star - T_\e(t- \sigma) \right] Y^\a_{\mu,3}(\si) d \sigma,}
 \end{array}
 \]
   where
  \[\le\{\begin{array}{l}
  \ds{Y^\a_{\mu,1}(\si):=\int_0^\sigma (\sigma - s)^{-\alpha} \left[ \Pi_1 S_\mu^\e(\sigma - s) Q_\mu - T_\e(\sigma - s) Q_\e \right] dw(s)}\\
  \vs
  \ds{Y^\a_{\mu,2}(\si):=  \int_0^\sigma (\sigma-s)^{-\alpha} \Pi_2 S_\mu^\e(\sigma-s) Q_\mu dw(s),}\\
  \vs
  \ds{Y^\a_{\mu,3}(\si):=\int_0^\sigma (\sigma - s)^{-\alpha} \Pi_1 S_\mu^\e(\sigma - s) Q_\mu dw(s).}
  \end{array}\r.\]

We have
\[\E \left| Y^\a_{\mu,1}(\sigma) \right|_{H}^2 \leq  \int_0^T s^{-2\alpha} \sum_{k=1}^\infty \left| \left[ \Pi_1 S_\mu^\e(s) Q_\mu - T_\e(s) Q_\e \right] e_k \right|_{H}^2 ds.\]
  If we choose $\alpha= \frac{\delta}{4}$,
  then by \eqref{bound-1}, \eqref{bound-2},  \eqref{Smu-to-Se-3-at-a-point}, and Hypothesis \ref{H5}, from  the dominated convergence theorem we have
    \[\lim_{\mu\to 0} \int_0^T  s^{-{\d/2}} \sum_{k=1}^\infty \left| \left[ \Pi_1 S_\mu^\e(s) Q_\mu - T_\e(s) Q_\e \right] e_k \right|_{H}^2 ds = 0.\]
Therefore, from the Gaussianity of $Y^{\d/4}_{\mu,1}(\sigma)$, for any $p\geq 2$
\[\lim_{\mu \to 0} \sup_{\sigma \leq T}\E|Y^{\d/4}_{\mu,1}(\sigma)|_{H}^p = 0.\]
Thanks to \eqref{Se-bound}, this implies that if we take $p$ large enough so that
$p(\d-4)/4(p-1)>-1$, we have
\[\begin{array}{l}
\ds{\lim_{\mu\to 0} \E\sup_{t \in\,[0,T]}\le|\int_0^t (t-\sigma)^{\d/4 -1} T_\e(t-\sigma)  Y^{\d/4}_{\mu,1}(\si)d\sigma\r|_H^p}\\
\vs
\ds{\leq  \sup_{0 \leq t \leq T} \|T_\e(t) \|_{{\cal L}(H)} \left( \int_0^T \sigma^{\frac{p(\d -4)}{4(p -1)}} d\sigma\right)^{p -1}\lim_{\mu\to 0} \int_0^T \E \left|Y_{\mu,1}^{\d/4}(\sigma) \right|_{H}^{p} d\sigma=0.}
\end{array}\]

Next, we remark that in view of \eqref{bound-3} and Hypothesis \ref{H5},
  \[\E |Y_{\mu,2}^{\d/4}(\sigma)|_{H^{\d -1}}^2 \leq \ c\,\mu^{-(1 + \frac{\d}2)}\sum_{k=1}^\infty \frac{\lambda_k^2} {\alpha_k^{1-\d}}<\infty,\]
and by Lemma \ref{Smu-bound-lem},
  \[\left| \Pi_1 S_\mu^\e(t) (0, Y_{\mu,2}^{\d/4}(\sigma)) \right|^2_H \leq 2^\d \mu^{1+\d} |Y_{\mu,2}^{\d/4}(\sigma) |^2_{H^{\d-1}}. \]
  Therefore,
  \[\sup_{\sigma \leq t \leq T} \E \left|\Pi_1 S_\mu^\e(t-\sigma) (0, \Pi_2 Y_{\mu,2}^{\d/4}(\sigma) )\right|_{H^{\d-1}}^{p} \leq
  c_p\,\mu^{\frac{p\d}4} \le(\sum_{k=1}^\infty \frac{\lambda_k^2}{\alpha_k^{1-\d}}\r)^{\frac p2}.\]
Therefore, if we pick again $p$ large enough so that $p(\d-4)/4(p-1)>-1$, we get
\[\begin{array}{l}
\ds{\lim_{\mu \to 0} \E \sup_{t \leq T} \le|\int_0^t (t-\sigma)^{\alpha-1} \Pi_1 S_\mu^\e(t-\sigma) (0,Y^\a_{\mu,2}(\si)) d\sigma\r|_H^p}\\
\vs
\ds{\leq T\left( \int_0^T \sigma^{\frac{(\delta-4) p}{4(p-1)}} d \sigma \right)^{p-1}\lim_{\mu\to 0}\sup_{\sigma\leq t\leq T} \E |\Pi_1 S_\mu^\e(t-\sigma) Y_{\mu,2}^{\d/4}(\sigma)|_{H^{\delta-1}}^{p} =0.}
\end{array}\]

  Finally, for any $n \in \nat$, we have
\[\begin{array}{l}
\ds{\int_0^t (t - \sigma)^{\frac{\d-4}4} \left[ \Pi_1 S_\mu^\e(t-\sigma)\Pi_1^\star - T_\e(t-\sigma) \right] Y_{\mu,3}^{\d/4}(\sigma) d\si}\\
\vs
\ds{= \int_0^t (t - \sigma)^{\frac{\d-4}4} \left[\Pi_1 S_\mu^\e(t-\sigma)\Pi_1^\star - T_\e(t-\sigma) \right] \left( P_nY_{\mu,3}^{\d/4}(\sigma)+ (Y_{\mu,3}^{\d/4}(\sigma) - P_nY_{\mu,3}^{\d/4}(\sigma))  \right)d \sigma.}
\end{array}
  \]
By \eqref{bound-1},
  \begin{equation*}
\sup_{\mu>0,\,n \in\,\nat}\E \left|P_nY_{\mu,3}^{\d/4}(\sigma) \right|_{H}^2 \leq \int_0^T s^{-\frac \d 2} \sum_{k=1}^\infty \left| S_\mu^\e(s) Q_\mu e_k \right|_{\H}^2 ds \leq c\,\sum_{k=1}^\infty \frac{\lambda_k^2}{\alpha_k^{1-\frac \d 2}}.
  \end{equation*}
  and
  \begin{equation*}
    \E \left| Y_{\mu,3}^{\d/4}(\sigma) - P_n Y_{\mu,3}^{\d/4}(\sigma) \right|_{H}^2 = \int_0^T \sigma^{-\frac \d 2} \sum_{k=n+1}^\infty \left| \Pi_1 S_\mu^\e(\sigma) Q_\mu e_k \right|_{H}^2 ds \leq c \sum_{k=n+1}^\infty \frac{\lambda_k^2}{\alpha_k^{1 - \frac\d 2}}.
  \end{equation*}
  This implies
\begin{equation}
\label{s6}
 \sup_{\mu>0, n \in \nat} \sup_{\si \in\,[0,T]}\E \left| P_n Y_{\mu,3}^{\d/4}(\sigma) \right|_{H}^{p} <+\infty
 \end{equation}
and
\begin{equation}
\label{s7}
\lim_{n \to +\infty} \sup_{\mu>0}\sup_{\si \in\,[0,T]}\E \left| Y_{\mu,3}^{\d/4}(\sigma) - P_n Y_{\mu,3}^{\d/4}(\sigma) \right|_{\H}^{p} = 0.\end{equation}
This implies that for any $p$ such that $p(\d-4)/4(p-1)>-1$
\[  \begin{array}{l}
  \ds{\E\sup_{t \in\,[0,T]}\le|\int_0^t (t - \sigma)^{\frac{\d-4}4} \left[ \Pi_1 S_\mu^\e(t-\sigma)\Pi_1^\star - T_\e(t-\sigma) \right] Y_{\mu,3}^{\d/4}(\sigma) d\si\r|_H^p}\\
  \vs
    \ds{\leq  \left(\int_0^T \sigma^{\frac{p(\d-4)}{4(p-1)}} d \sigma \right)^{p - 1}\le( \sup_{t \in\,[0,T]} \sup_{|x|_{H} \leq 1} |\Pi_1 S_\mu^\e(t)(P_n x,0) - T_\e(t)P_n x  |_H^{p}
  \int_0^T \E\left| P_n Y_{\mu,3}^{\d/4}(\sigma) \right|_{H}^{p} d\sigma\r.} \\
   \vs
   \ds{ \le.+ \sup_{t  \in\,[0,T]} \left( \|\Pi_1 S_\mu^\e(t) \Pi_1^\star \|_{{\cal L}(H)} + \|T_\e(t) \|_{{\cal L}(H)} \right) \int_0^T \E\left|Y_{\mu,3}^{\d/4}(\sigma) - P_n Y_{\mu,3}^{\d/4}(\sigma) \right|^{p}_\H d \sigma\r)}
  \end{array}
  \]
By  choosing $n$ large enough, \eqref{s6} and \eqref{s7} yield
\[\lim_{\mu \to 0}  \E  \sup_{t \in\,[0,T]} \le|\int_0^t (t - \sigma)^{\frac{\d-4}4} \left[ \Pi_1 S_\mu^\e(t-\sigma)\Pi_1^\star - T_\e(t-\sigma) \right] Y_{\mu,3}^{\d/4}(\sigma) d\si\r|_H^p=0.\]

\end{proof}

\section{Approximation by small friction for multiplicative noise } \label{sec5}

In this section we assume that the space dimension $d=1$ and $D$ is a bounded interval, the diffusion coefficient $G$ satisfies Hypothesis \ref{H3} and the covariance operator $Q$ satisfies the following condition.
\begin{Hypothesis}
\label{H7}
There exists a bounded non-negative sequence $\{\la_k\}_{k \in\,\nat}$ such that
\[Q e_k=\la_k e_k,\ \ \ \ k \in\,\nat.\]
\end{Hypothesis}

We begin by studying the stochastic convolutions
\[
\Gamma_\mu^\e(z)(t) := \int_0^t S_\mu^\e(t-s) G_\mu(z(s),s)dw^Q(s),\ \ \ \ z \in\,L^p(\Omega,C([0,T];\H)),\]
and
\[\Gamma_\e(u)(t) = \int_0^t T_\e(t-s) G_\e(u(s),s) dw^Q(s), u \in\,L^p(\Omega,C([0,T];H)).\]
With the notations introduced in Sections \ref{s2} and \ref{s3}, the regularized system \eqref{regular-intro} can be rewritten as
\begin{equation}
\label{s9}
dz^\e_\mu(t)=\le[A^\e_\mu z^\e_\mu(t)+B_\mu(z^\e_\mu(t),t)\r]\,dt+G_\mu(z^\e_\mu(t),t)\,dw^Q(t),\ \ \ \ z^\e_\mu(0)=(u_0,v_0),
\end{equation}
and the limiting problem \eqref{first-order-eq-w-frict} can be rewritten as
\begin{equation}
\label{s10}
du_\e(t)=\le[A_\e u_\e(t)+B_\e (u_\e(t),t)\r]\,dt+G_\e(u_\e(t),t)\,dw^Q(t),\ \ \ \ \ u_\e(0)=u_0,\end{equation}
where
\[G_\e(u,t)=J_\e^{-1}G(u,t).\]

\begin{Lemma} \label{lem:stoch-conv-Lip}
Under Hypotheses \ref{H3} and \ref{H7}, for any $\mu, \e>0$, $T \geq 0$ and $p>4$ we have
  \[z \in\,L^p(\Omega;C([0,T];\H))\Longrightarrow \Gamma_\mu^\e(z) \in L^p(\Omega;C([0,T];\H)). \]
Moreover, there exists a constant $c:=c(\e,\mu,p,T)$ such that
  \begin{equation}
  \label{s15}
  \E | \Gamma_\mu^\e(z_1) - \Gamma_\mu^\e(z_2) |^p_{C([0,T];\H)} \leq c_p\, \int_0^T\E |\Pi_1 z_1-\Pi_1 z_2|^p_{C([0,\si];H)}\,d\si.\end{equation}
\end{Lemma}
\begin{proof}
It is sufficient to prove \eqref{s15}. By the factorization method, for any $\alpha \in (0,1/2)$ we have
  \[\begin{array}{l}
  \ds{\Gamma_\mu^\e(z_1)(t) - \Gamma_\mu^\e(z_2)(t) = \frac{\sin(\pi\alpha)}{\pi} \int_0^t (t-\sigma)^{\alpha-1}  S_\mu^\e(t-\sigma) Y^\mu(\sigma)d\sigma,}\end{array}\]
  where
  \[Y^\mu(\sigma)=:(Y_1^\mu(\sigma),Y_2^\mu(\sigma)),\]
  with
   \[Y^\mu_1(\sigma) = \int_0^\sigma (\sigma - s)^{-\alpha} \Pi_1 S_\mu^\e(\sigma - s) \le[G_\mu(z_1(s),s) - G_\mu(z_2(s),s)\r]dw^Q(s)\]
  and
  \[Y^\mu_2(\sigma) = \int_0^\sigma (\sigma -s)^{-\alpha} \Pi_2 S_\mu^\e(\sigma - s) \le[G_\mu(z_1(s),s) - G_\mu(z_2(s),s)\r]  dw^Q(s).\]
  Then, for any $p>1/\a$ we have
  \begin{equation}
  \label{s14}
\begin{array}{l}
\ds{  \le|\Gamma_\mu^\e(z_1)(t) - \Gamma_\mu^\e(z_2)(t)\r|_\H^p\leq c_{\mu,\e,p} \le(\int_0^T \sigma^{\frac{(\a-1)p}{p-1}}\,d\sigma\r)^{p-1}\int_0^t\le(|Y^\mu_1(\sigma)|_H^p+|Y^\mu_2(\sigma)|_{H^{-1}}^p\r)\,d\si.}
\end{array}
\end{equation}

By the Burkholder-Davis-Gundy inequality, we have
  \[\begin{array}{l}
  \ds{ E |Y^\mu_1(\sigma)|_{H}^{p} }\\
  \vs
  \ds{\leq \frac{c_p}{\mu^p}\, \E \left( \int_0^\sigma (\sigma -s)^{-2\alpha} \sum_{k=1}^\infty \la_k^2\left|\Pi_1 S_\mu^\e(\sigma -s)(0, [G(\Pi_1 z_1(s),s) - G(\Pi_1 z_2(s),s)]  e_k )\right|_{H}^2 ds\right)^{\frac p2}.}
  \end{array}\]
Now, for any $v \in\,H^{-1}$ we have
\[\begin{array}{l}
\ds{\Pi_1 S_\mu^\e(t)(0,v)}\\
\vs
\ds{=\sum_{h=1}^\infty \le[\le<\Pi_1 S^\e_\mu(t)(0,e_{2h-1}),e_{2h-1}\r>_H\le<v,e_{2h-1}\r>_H+\le<\Pi_1 S^\e_\mu(t)(0,e_{2h}),e_{2h-1}\r>_H\le<v,e_{2h}\r>_H\r]e_{2h-1}}\\
\vs
\ds{+\sum_{h=1}^\infty \le[\le<\Pi_1 S^\e_\mu(t)(0,e_{2h-1}),e_{2h}\r>_H\le<v,e_{2h-1}\r>_H+\le<\Pi_1 S^\e_\mu(t)(0,e_{2h}),e_{2h}\r>_H\le<v,e_{2h}\r>_H\r]e_{2h}.}
\end{array}\]
This easily implies
\[
\begin{array}{l}
\ds{|\Pi_1 S_\mu^\e(t)(0,v)|_H^2\leq c\sum_{h=1}^\infty\le|\Pi_1 S^\e_\mu(t)(0,e_{h})\r|_H^2\le|\le<v,e_{h}\r>_H\r|^2,}
\end{array}\]
so that
\[\begin{array}{l}
\ds{\sum_{k=1}^\infty \left|\Pi_1 S_\mu^\e(\sigma -s)(0, [G(\Pi_1z_1(s),s) - G(\Pi_1z_2(s),s)]  e_k )\right|_{H}^2}\\
\vs
\ds{\leq c\,\sum_{k=1}^\infty \sum_{h=1}^\infty \le|\Pi_1 S^\e_\mu(\sigma -s)(0,e_{h})\r|_H^2\le|\le<[G(\Pi_1z_1(s),s) - G(\Pi_1z_2(s),s)]  e_k,e_h\r>_H\r|^2}\\
\vs
\ds{=
\sum_{h=1}^\infty \le|\Pi_1 S^\e_\mu(\sigma -s)(0,e_{h})\r|_H^2 \sum_{k=1}^\infty \le|\le<[G^\star(\Pi_1z_1(s),s) - G^\star(\Pi_1z_2(s),s)]  e_h,e_k\r>_H\r|^2}\\
\vs
\ds{=\sum_{h=1}^\infty \le|\Pi_1 S^\e_\mu(\sigma -s)(0,e_{h})\r|_H^2  \le|[G^\star(\Pi_1z_1(s),s) - G^\star(\Pi_1z_2(s),s)]  e_h \r|_H^2.}
\end{array}\]
Therefore, thanks to \eqref{s12} and \eqref{bound-1}, for any $\a<1/2$ we get
\[\begin{array}{l}
\ds{E |Y^\mu_1(\sigma)|_{H}^{p} \leq \frac{c_{p,T}}{\mu^p}\, \E \left( \int_0^\sigma (\sigma -s)^{-2\alpha} \sum_{h=1}^\infty \le|\Pi_1 S^\e_\mu(\si-s)(0,e_{h})\r|_H^2|\Pi_1z_1(s)-\Pi_1z_2(s)|_{H}^2 ds\right)^{\frac p2}}\\
\vs
\ds{\leq \frac{c_{p,T}}{\mu^p}\E\,|\Pi_1 z_1-\Pi_1 z_2|^p_{C([0,T];H)}\le(\int_0^\sigma s^{-2\alpha} \sum_{h=1}^\infty \le|\Pi_1 S^\e_\mu(s)(0,e_{h})\r|_H^2 ds\right)^{\frac p2}}\\
\vs
\ds{ \leq c_{p,T}\,\E\,|\Pi_1 z_1-\Pi_1 z_2|^p_{C([0,T];H)}\le(\sum_{h=1}^\infty\frac 1{\a_h^{1-2\a}}\r)^{\frac p2},}
\end{array}\]
and if we take  $\a<1/4$, we can conclude that
  \begin{equation} \label{eq:Y_1^2m-bound-mult}
    \E |Y^\mu_1(\sigma)|_{H}^{p} \leq c_{p,T}\, \E |\Pi_1 z_1-\Pi_1 z_2|_{C([0,\sigma];H)}^{p}.
  \end{equation}
By proceeding in the same way as for $Y^\mu_1(\si)$, for any $\theta<1/2$ we have
  \[\begin{array}{l}
  \ds{\E |Y^\mu_2(\sigma)|_{H^{\theta-1}}^{p} \leq \frac{c_{p,T}}{\mu^p}\E\,|\Pi_1 z_1-\Pi_1 z_2|^p_{C([0,\si];H)}\le(\int_0^\sigma s^{-2\alpha} \sum_{h=1}^\infty \le|\Pi_2 S^\e_\mu(s)(0,e_{h})\r|_{H^{\theta-1}}^2 ds\right)^{\frac p2},}
  \end{array}\]
  and then, thanks to
\eqref{bound-3}
  \begin{equation}\label{eq:Y_2^2m-bound-mult}
  \E |Y^\mu_2(\sigma)|_{H^{-1}}^p \leq c_{p,T}\,\mu^{-\frac{p(1 + 2\alpha)}{2}}\, \E |\Pi_1 z_1-\Pi_1 z_2|_{C([0,\si];H)}^{p}.
  \end{equation}
Therefore, according to \eqref{s14}, if $p>4$ we can find $\a_p \in\,(1/p,1/4)$ such that
\[\E\,\le|\Gamma_\mu^\e(z_1) - \Gamma_\mu^\e(z_2)\r|_{C([0,T];\H)}^p\leq c_p\,\int_0^T\E\,|\Pi_1 z_1-\Pi_1 z_2|_{C([0,\si];H)}^p\,d\si,\]
for a constant $c$ depending on $p,\ \mu,\ \e$ and $T$.
\end{proof}

\begin{Remark}
{\em From the proof of the Lemma above, we easily see that, as a consequence of estimates \eqref{energy-est-1} and \eqref{s2}, for any $z_1, z_2 \in\,L^p(\Omega;C([0,T];\H))$
  \begin{equation}
  \label{s30}
  \sup_{\mu>0} \E | \Pi_1\Gamma_\mu^\e(z_1) - \Pi_1\Gamma_\mu^\e(z_2) |^p_{C([0,T];\H)} \leq c_p\, \int_0^T\E |\Pi_1 z_1-\Pi_1 z_2|^p_{C([0,\si];\H)}\,d\si
,\end{equation}
for a constant $c=c(\e,p,T)>0$.}
\end{Remark}

In Lemma \ref{lem:stoch-conv-Lip} we have proven that the mapping
\[z \in\,L^p(\Omega;C([0,T];\H))\mapsto \Gamma_\mu^\e(z) \in\,L^p(\Omega;C([0,T];\H)),\]
is Lipschitz continuous. Therefore, as the mapping $B_\mu(\cdot,t):\H\to \H$, is Lipschitz continuous, uniformly for $t \in\,[0,T]$, we have that for any initial condition $z_0=(u_0,v_0) \in\,\H$, system \eqref{s9} admits uniqued adapted mild solution $z_\mu^\e \in\,L^p(\Omega;C([0,T];\H))$.

\begin{Lemma}
\label{lem:stoch-conv-epsilon-Lip}
 Under Hypotheses \ref{H3} and \ref{H7}, for any $\e, T \geq 0$ and any $p>4$
   \[u \in\,L^p(\Omega;C([0,T];H)\Longrightarrow \Gamma_\e(u) \in L^p(\Omega;C([0,T];H)). \]
   Moreover,
there exists a constant $c:=c(\e,p,T)$ such that for any $u,v \in\,L^p(\Omega;C([0,T];H))$
 \begin{equation}
 \label{s46}
 \E \left| \Gamma_\e(u) - \Gamma_\e(v) \right|_{C([0,T];H)}^{p} \leq c\, \int_0^T\E|u-v|_{C([0,\si];H)}^{p}\,d\si.\end{equation}
 If we assume that
 \[\sum_{k=1}^\infty \la_k^2<\infty,\]
 then the constant $c$ in \eqref{s46} is independent of $\e>0$.
\end{Lemma}

\begin{proof}
The proof is obtained from the same arguments used in the proof of Lemma \ref{lem:stoch-conv-Lip}, just by replacing the use of Lemmas \ref{energy-est-lem}, \ref{Smu-bound-lem} and \ref{l1}, with the use of Lemma \ref{l7}.
\end{proof}

As a consequence of this lemma, since the mapping $B_\e(\cdot,t):H\to H$ is Lipschitz continuous, uniformly for $t \in\,[0,T]$, we have that for any initial condition $u_0 \in\,\H$, system \eqref{s9} admits uniqued adapted mild solution $u_\e \in\,L^p(\Omega;C([0,T];H))$.

\begin{Theorem} \label{thm:Gamma-mu-to-Gamma-e}
  For any  fixed $\e>0$, $T>0$ and $p\geq 1$, there exists $c:=c(T,\e,p)$ such that for any $u \in L^{p}(\Omega, C([0,T];H))$
  \[\lim_{\mu \to 0} \E \left|\Pi_1\Gamma_\mu^\e((u,0)) - \Gamma_\e(u) \right|_{C([0,T];H)}^{p}=0.\]
\end{Theorem}
\begin{proof}
Once again, by the factorization method,
 we can write
  \[\begin{array}{l}
  \ds{\frac \pi{\sin(\pi \a)}\le[\Pi_1\Gamma_\mu^\e(u,0)(t) - \Gamma_\e(u)(t)\r]=\int_0^t (t-\sigma)^{\alpha -1} T_\e(t-\sigma)Y^\mu_1(\sigma) d\sigma}\\
  \vs
  \ds{+ \int_0^t (t-\sigma)^{\alpha-1} [\Pi_1 S_\mu^\e(t-\sigma) \Pi_1^\star - T_\e(t-\sigma)] Y^\mu_2(\sigma) d\sigma}\\
  \vs
  \ds{+ \int_0^t (t-\sigma)^{\alpha-1} \Pi_1 S_\mu^\e(t-\sigma) \left(0, Y^\mu_3(\sigma) \right)  d\sigma}\\
  \vs
  \ds{:= I^\mu_1(t) + I^\mu_2(t) + I^\mu_3(t).}
  \end{array}\]
  where
  \[\begin{array}{l}
  \ds{Y^\mu_1(\sigma) = \int_0^\sigma (\sigma-s)^{-\alpha}[\Pi_1 S_\mu^\e(\sigma-s) G_\mu((u(s),0),s) - T_\e(\sigma-s) G_\e(u(s),s)]dw^Q(s) }\\
  \ds{Y^\mu_2(\sigma) = \int_0^\sigma (\sigma-s)^{-\alpha} \Pi_1 S_\mu^\e(\sigma -s) G_\mu((u(s),0),s)  dw^Q(s)}\\
  \ds{Y^\mu_3(\sigma) = \int_0^\sigma (\sigma - s)^{-\alpha} \Pi_2 S_\mu^\e(\sigma -s) G_\mu((u(s),0),s)  dw^Q(s)}
  \end{array}\]
  By the Burkholder-Davis-Gundy inequality, for any $p\geq 2$
  \[\begin{array}{l}
    \ds{ \E |Y_1^\mu(\sigma)|_{H}^{p}}\\
    \vs
    \ds{ \leq c_p\,\E \left( \int_0^\sigma (\sigma-s)^{-2\alpha} \sum_{k=1}^\infty \le| [\Pi_1 S_\mu^\e(\sigma-s)G_\mu((u(s),0),s) - T_\e(\sigma -s) G_\e(u(s),s)]e_k \right|_{H}^2 ds  \right)^{\frac p2}.}
  \end{array}\]
Now, proceeding as in the proof of Lemma \ref{lem:stoch-conv-Lip}, we have
  \[\begin{array}{l}
  \ds{\sum_{k=1}^\infty \int_0^\sigma (\sigma -s)^{-2\alpha} \left| [\Pi_1 S_\mu^\e(\sigma-s)G_\mu((u(s),0),s) - T_\e(\sigma -s) G_\e(u(s),s)]e_k \right|_H^2 }\\
  \vs
  \ds{ \leq c_{\a,T}\, ( 1 + |u|_{C([0,T];H)}) \sum_{h=1}^\infty \frac 1{\a_h^{1-2\a}}.}
  \end{array}\]
Moreover, according to
 \eqref{Smu-to-Se-3-at-a-point}, for any  fixed $0\leq s < \sigma$ and $k \in \nat$
  \[\lim_{\mu \to 0} \left| [\Pi_1 S_\mu^\e(\sigma-s)G_\mu((u(s),0),s) - T_\e(\sigma -s) G_\e(u(s),s)]e_k \right|_{H} = 0,\ \ \ \Pro-\text{a.s.}\]
  Therefore, by the dominated convergence theorem for any $\sigma \leq T$,
  \[\lim_{\mu \to 0} \E|Y_1^\mu(\sigma)|_H^{p} =0,\]
so that, if $p>4$ there exists $\a \in\,(1/p,1/4)$  such that
  \[|I^\mu_1|_{C([0,T];H)}^{p} \leq \left(\int_0^T \sigma^{\frac{p(\alpha-1)}{p-1}} d\sigma \right)^{p-1} \sup_{s\geq0} \|T_\e(s)\|_{{\cal L}(H)}^{p}  \int_0^T |Y^\mu_1(\sigma)|_H^{p} d\sigma.\]
Due to the dominated convergence theorem, we can conclude that
  \begin{equation}
  \label{s18}
  \lim_{\mu \to 0} \E|I^\mu_1|_{C([0,T];H)}^{p} = 0.\end{equation}

For each $n \in\,\nat$ we rewrite $I^\mu_2(t)$ as
  \[\begin{array}{l}
  \ds{I^\mu_2(t) = \int_0^t (t-\sigma)^{1-\alpha} [\Pi_1 S_\mu^\e(t-\sigma) \Pi_1^\star - T_\e(t-\sigma)] (Y^\mu_2(\sigma) - P_n Y^\mu_{2}(\sigma)) d\sigma}\\
  \vs
  \ds{+\int_0^t (t-\sigma)^{1 - \alpha} [\Pi_1 S_\mu^\e(t-\sigma) \Pi_1^\star - T_\e(t-\sigma)]P_n Y^\mu_{2}(\sigma) d\sigma.}
  \end{array}\]
Therefore, if $\a<1/p$
  \[\begin{array}{l}
    \ds{|I^\mu_2|_{C([0,T];H)}^{p} \leq c_p \left( \int_0^T (T-\sigma)^{\frac{p(\alpha-1)}{p-1}} d\sigma \right)^{p-1}}\\
    \vs
     \ds{\le[\,\sup_{t\geq0} \left(\|\Pi_1 S_\mu^\e(t)\Pi_1^\star\|^p_{{\cal L}(H)} + \|T_\e(t)\|^p_{{\cal L}(H)} \right)\int_0^T |Y^\mu_2(\sigma) - P_n Y^\mu_{2}(\sigma)|_H^{p} d\sigma \r.}\\
     \vs
    \ds{\le.+  \sup_{0 \leq t \leq T} \sup_{|x|_H\leq 1} \left|S_\mu^\e(s)\Pi_1^\star P_n x - T_\e(s)P_n x \right|_H^{p} \int_0^T |P_n Y^\mu_{2}(\sigma)|_H^{p} d \sigma \right]:= I^\mu_{2,1} (n)+ I^\mu_{2,2}(n).}
  \end{array}\]

  Now, as we have seen above for $Y^\mu_1(t)$, for any $n \in\,\nat$ we have
    \begin{equation}
    \label{s20}
    \begin{array}{l}
  \ds{\E|P_n Y^\mu_{2}(\sigma)|_H^{p} \leq c_{p,T}\,  \left(1 + \E\,|u|_{C([0,T];H)}^{p} \right).}
  \end{array}\end{equation}
Moreover, by proceeding as in the proof of Lemma \ref{lem:stoch-conv-Lip}, we have  By the same arguments, we can show that
  \[\begin{array}{l}
  \ds{\E \left|Y^\mu_2(\sigma) - P_n Y^\mu_{2}(\sigma) \right|_H^{p}}\\
  \vs
  \ds{ \leq \frac{c_{p,T}}{\mu^p}\le(1 + \E\,|u|_{C([0,T];H)}^{p} \right)\le(\int_0^\si s^{-2\a}\sum_{h=n+1}^\infty |\Pi_1 S^\e_\mu(s)(0,e_h)|_H^2\,ds\r)^{\frac p2}}\\
  \vs
  \ds{\leq c_{p,T}\,\le(1 + \E\,|u|_{C([0,T];H)}^{p} \right)\le(\sum_{h=n+1}^\infty\frac 1{\a_h^{1-2\a}}\r)^{\frac p2}.}
  \end{array}\]
  Therefore, if $\a<1/4$ we get
  \begin{equation} \label{eq:Y_2-to-Y_2,N}
  \lim_{n \to \infty } \E \left|Y^\mu_2(\sigma) - P_n Y^\mu_{2}(\sigma) \right|_H^{p}=0.
  \end{equation}
Since by \eqref{energy-est-1}, $\|\Pi_1 S_\mu^\e(s) \Pi_1^\star \|_{{\cal L}(H)}$ is uniformly bounded independently of $s$ and $\mu$, this means that for any $\eta>0$there exists  $n_\eta \in\,\nat$ such that
  \[\sup_{\mu>0} \E I^\mu_{2,1}(n_\eta) < \frac{\eta}{2}.\]
  By \eqref{Smu-to-Se-1} and \eqref{s20} we can find $\mu_0>0$ small enough such that
  \[\E I^\mu_{2,2}(n_\eta) < \frac{\eta}{2},\ \ \ \ \mu\leq \mu_0,\]
and then,   since $\eta$ was arbitrary, we can conclude that
  \begin{equation}
  \label{s19}
  \lim_{\mu \to 0} \E |I^\mu_2|_{C([0,T];H)}^{p} = 0.\end{equation}

  It remains to estimate $I^\mu_3(t)$.
By proceeding as in the proof of Lemma \ref{lem:stoch-conv-Lip}, we have
  \[\begin{array}{l}
    \ds{|Y^\mu_3(\sigma)|_H^{p} \leq c_{p,T} \mu^{-\frac{p(1 + 2\alpha)}{2}} \E \left( 1 + |u|_{C([0,T];H)}^{p} \right).}
  \end{array}\]
Then, for $\a<1/p$ we have
  \[|I^\mu_3|_{C([0,T];H)}^{p} \leq \left(\int_0^T \si^{\frac{2m(\alpha-1)}{2m-1}} d\sigma \right)^{p-1} \sup_{t\geq 0} \|\Pi_1 S_\mu^\e(t) \Pi_2^\star\|_{{\cal L}(H)}^{p} \left( \int_0^T |Y^\mu_3(\sigma)|_H^{p} d\sigma\right).\]
  From Lemma \ref{Smu-bound-lem},
  \[\|\Pi_1 S_\mu^\e(s) \Pi_2^\star\|_{{\cal L}(H)}^{p} \leq c \mu^{p}.\]
  Therefore,
  \[\E|I^\mu_3|_{C([0,T];H)}^{2m} \leq c \mu^{\frac{p(1 - 2\alpha)}{2}} \E (1 + |u|^{p}_{C([0,T];H)}),\]
  and we can conclude that
  \[\lim_{\mu \to 0} \E|I^\mu_3|_{C([0,T];H)}^{p} = 0.\]

This, together with \eqref{s18} and \eqref{s19} implies that for any $p>4$
  \[\lim_{\mu\to 0}\E |\Pi_1\Gamma_\mu^\e(u,0) - \Gamma_\e(u)|_{C([0,T];H)}^{p} = 0.\]
  The case $p\geq 1$ is a consequence of the H\"older inequality.
\end{proof}

\begin{Theorem}
\label{ts31}
Let $z^\e_\mu=(u^\e_\mu,v^\e_\mu)$ and $u_\e$ be the mild solutions of problems \eqref{s9} and \eqref{s10}, with initial conditions $z_0 \in\,\H$ and $u_0=\Pi_1 z_0 \in\,H$, respectively.
 Then, under Hypotheses \ref{H2}, \ref{H3} and \ref{H7}, for any $T>0$,  $\e>0$ and $p\geq 1$ we have
  \[\lim_{\mu \to 0}\E |u^\e_\mu - u_\e|_{C([0,T];H)}^{p}=0.\]
\end{Theorem}
\begin{proof}
We have
  \[u^\e_\mu(t) = \Pi_1 S_\mu^\e(t)(u_0,v_0) + \Pi_1 \int_0^t S_\mu^\e(t-s) B_\mu(z_\mu^\e(s),s) ds + \Pi_1 \Gamma_\mu^\e(z^\e_\mu)(t), \]
and
  \[u_\e(t) = T_\e(t)u_0 + \int_0^t T_\e(t-s) B_\e(u_\e(s),s) + \Gamma_\e(u_\e)(t).\]
  Then
  \[\begin{array}{l}
  \ds{\left|u_\mu^\e(t) - u_\e(t) \right|_H \leq \left| \Pi_1 S_\mu^\e(t) (u_0,v_0) - T_\e(t)u_0\right|_H}\\
  \vs
  \ds{ + \left| \int_0^t \Pi_1 S_\mu^\e(t-s)[B_\mu(z^\e_\mu(s),s) - B_\mu((u_\e(s),0),s)] ds \right|_H }\\
  \vs
  \ds{ + \left| \frac 1\mu\int_0^t \Pi_1 S_\mu^\e(t-s) (0,B(u_\e(s),s))ds - \int_0^t T_\e(t-s) J_\e^{-1} B(u_\e(s),s) ds \right|_H}\\
  \vs
   \ds{+ \left|\Pi_1\le[\Gamma_\mu^\e(z^\e_\mu)(t) - \Gamma_\mu^\e((u_\e(t),0))\r] \right|_{H} + \left|\Pi_1 \Gamma_\mu^\e(u_\e(t),0) - \Gamma_\e(u_\e)(t) \right|_H.}
  \end{array}\]
 By Lemma \ref{Smu-bound-lem}, and Hypothesis \ref{H2}, there is a constant independent of $\mu$ and of  $0<s<t$, such that
  \[\left|\Pi_1 S_\mu^\e(t-s)[B_\mu(z^\e_\mu(s),s) -B_\mu((u_\e(s),0),s)] \right|_H \leq c\, |u_\mu^\e(s) - u_\e(s)|_H, \]
  so that for any $p\geq 2$
  \[\left|\int_0^t \Pi_1 S_\mu^\e(t-s)[B_\mu(z^\e_\mu(s),s) - B_\mu((u_\e(s),0),s)] ds\right|_H^{p}
  \leq c_p\, t^{p-1} \int_0^t |u^\e_\mu-u_\e|_{C([0,s];H)}^{p} ds. \]
Thanks to \eqref{s30}, this implies
  \[\begin{array}{l}
    \ds{\E\left|u^\e_\mu - u_\e \right|_{C([0,t];H)}^{p} \leq c_p\, T^{p-1} \int_0^t \E \left| u^\e_\mu - u_\e \right|_{C([0,s];H)}^{p} ds}\\
    \vs
     \ds{+c_p\,\sup_{s\leq t} \left| \Pi_1 S_\mu^\e(s) (u_0,v_0) - T_\e(t)u_0\right|_H^{p} + c_p\E\left| \Pi_1\Gamma_\mu^\e((u_\e,0)) -
     \Gamma_\e(u_\e) \right|_{C([0,t];H)}^{p} }\\
     \vs
  \ds{ +c_p\,\E\sup_{s \leq t} \left|\frac 1\mu\int_0^s \Pi_1 S_\mu^\e(s-r) (0,B(u_\e(r),r))dr - \int_0^s T_\e(s-r) B_\e(u_\e(r),r) dr \right|_H^{p},}
    \end{array}\]
and the Gr\"onwall's inequality yields
  \[\begin{array}{l}
    \ds{\E \left| u^\e_\mu - u_\e \right|_{C([0,T];H)}^{p} }\\
    \vs
 \ds{   \leq c_p(T)\le( \sup_{s\leq T} \left| \Pi_1 S_\mu^\e(s) (u_0,v_0) - T_\e(t)u_0\right|_H^{p} + \E\left| \Pi_1\Gamma_\mu^\e((u_\e,0)) -
     \Gamma_\e(u_\e) \right|_{C([0,T];H)}^{p} \r)}\\
     \vs
  \ds{ +c_p(T)\,\E\sup_{s \leq T} \left|\frac 1\mu\int_0^s \Pi_1 S_\mu^\e(s-r) (0,B(u_\e(r),r))dr - \int_0^s T_\e(s-r) B_\e(u_\e(r),r) dr \right|_H^{p}.}
  \end{array}\]
  Finally, the result follows because of \eqref{Smu-to-Se-at-a-point}, \eqref{Smu-to-Se-2}, and Theorem \ref{thm:Gamma-mu-to-Gamma-e}.
\end{proof}

\section{The convergence for $\e \downarrow 0$ } \label{sec6}

In the previous sections, we have shown that under suitable conditions on the coefficients and the noise, for any fixed $\e>0$, $T>0$ and $p\geq 1$
\[\lim_{\mu \to 0} \E \left|u^\e_\mu-u_\e  \right|_{C([0,T];H)}^p =0. \]
This limit is not uniform in $\e>0$, and the limit is not true for $\e=0$. In this section we want  to show that
\begin{equation}
\label{limite}
\lim_{\e \to 0} \E \left|u_\e - u \right|_{C([0,T];H)}^p = 0,\end{equation}
where $u$ is the mild solution of the problem
\begin{equation}
\label{s33}
du(t)=\le[A_0 u(t)+B_0(u(t),t)\r]\,dt+G_0(u(t),t)\,dw^Q(t),\ \ \ \ \ u(0)=u_0,\end{equation}
with
\[A_0:=J_0^{-1} A,\ \ \ B_0=J_0^{-1} B,\ \ \ \ G_0=J_0^{-1} G.\]

This statement is true  if we strengthen Hypothesis \ref{H5}. Actually, Hypothesis \ref{H5} is the weakest assumption on the regularity of the noise that implies Theorem \ref{conv-in-mu-thm} and Theorem \ref{ts31}, for $\e>0$. But in order to prove \eqref{limite} we need to assume the following stronger condition on the covariance $Q$.

\begin{Hypothesis} \label{H6}
There exists a non-negative sequence $\{\la_k\}_{k \in\,\nat}$ such that   $Q e_k = \lambda_k e_k$, for any $k \in\,\nat$, and
   \[\sum_{k=1}^\infty \lambda_k^2 < +\infty.\]
\end{Hypothesis}

In what follows, we shall denote by $T_0(t)$, $t\geq 0$, the semigroup generated by the differential operator $A_0$ in $H$, with $D(A_0)=D(A)$.
The semigroup $T_0(t)$ is strongly continuous in $H$. Moreover, if we define $u(t)=T_0(t)x$, for $x \in\,D(A_0)$, we have
\[\le\{\begin{array}{l}
\ds{\frac{\partial u_1}{\partial t}(t)=-\Delta u_2(t),\ \ \ \ u_1(0)=x_1}\\
\vs
\ds{\frac{\partial u_2}{\partial t}(t)=\Delta u_1(t),\ \ \ \ u_2(0)=x_2}
\end{array}\r.\]
This means that if we take the scalar product in $H^\theta$ of the first equation by $u_1$ and of the second equation by $u_2$, we get
\[\frac d{dt}|u(t)|^2_{H^\theta}=0,\]
so that
\begin{equation}
\label{s32}
|T_0(t)x|_{H^\theta}=|x|_{H^\theta},\ \ \ \ t\geq 0,\end{equation}
for any $\theta \in\,\reals$ and $x \in\,H$.

Now, let us consider the stochastic convolution associated with problem \eqref{s33}, in the simple case $G=I$
\[\Gamma(t)=\int_0^t T_0(t-s)Qdw(s),\ \ \ \ t \geq0.\]
As a consequence of \eqref{s32}, we have
\[\E\,|\Gamma(t)|_H^2=\int_0^t\sum_{k=1}^\infty |T_0(s)Qe_k|_H^2\,ds=\int_0^t\sum_{k=1}^\infty |Qe_k|_H^2\,ds=t \sum_{k=1}^\infty \la_k^2,\]
and this implies that Hypothesis \ref{H6} is necessary in order to have a solution in $H$ for the limiting equation \eqref{s33}.

\begin{Lemma}
  The matrix $J_\e^{-1}$ converges to $J_0^{-1}$ in ${\cal L}(\mathbb{R}^2)$. Furthermore,  for any $T\geq0$,
  \begin{equation} \label{A_e^inv-conv-1}
  \lim_{\e \to 0} \sup_{t \in\,[-T,T]} \left\|e^{tJ_\e^{-1}} - e^{t J_0^{-1}} \right\|_{{\cal L}(\mathbb{R}^2)} =0
  \end{equation}
\end{Lemma}
\begin{proof}
Recall that
 \[J_\e^{-1} = \frac{1}{1+ \e^2} \begin{pmatrix} \e & -1 \\ 1 & \e \end{pmatrix}=\frac{\e}{1+\e^2} I+\frac{1}{1+\e^2} J_0^{-1}.\]
 Then,
 \[J_\e^{-1} - J_0^{-1} = \frac{\e}{\e^2+1}I-\frac{\e^2}{1+\e^2}J_0^{-1}\]
 and this means that
 \[\|J_\e^{-1} - J_0^{-1}\|_{{\cal L}(\mathbb{R}^2)} \leq c\,\frac{\e}{\e^2+1}.\]
Moreover, we have
 \[e^{tJ_\e^{-1}} = e^{\frac{\e t}{\e^2+1}} \begin{pmatrix}
   \cos\,\frac{t}{\e^2+1} & -\sin\,\frac{t}{\e^2+1}\\
   \sin\,\frac{t}{\e^2+1} & \cos\,\frac{t}{\e^2+1}
 \end{pmatrix}.\]
 Therefore, limit \eqref{A_e^inv-conv-1} follows and is uniform  with respect to $t \in [-T,T]$.
\end{proof}

\begin{Lemma}
   For any $n \in \nat$, and $T\geq0$,
   \begin{equation} \label{Se-to-S0-1}
     \lim_{\e \to 0}\,\sup_{t  \in\,[0,T]} \sup_{|x|_H \leq 1} \left|T_\e(t)P_{n} x - T_0(t)P_n x \right|_H = 0.
   \end{equation}
\end{Lemma}

\begin{proof}
  If $x \in \span\{e_{2k-1}, e_{2k}\}$, then $T_\e(t)x = e^{-\alpha_{2k} J_\e^{-1}t} x$ and $T_0(t)x = e^{-\alpha_{2k} J_0^{-1}t} x$.
  Therefore, by \eqref{A_e^inv-conv-1},
  \[\lim_{\e \to 0}\, \sup_{t  \in\,[0,T]} \sup_{|x|_H \leq 1} \left| T_\e(t)P_n x - T_0(t)P_nx \right|_H \leq \lim_{\e \to 0} \,\sup_{t \in\,[-T \alpha_{2k},0]} \left\|e^{tJ_\e^{-1}} - e^{t J_0^{-1}} \right\|_{{\cal L}(\mathbb{R}^2)} = 0.\]
As we can extend this result to $\span\{e_k\}_{k=1}^{2n}$, for any $n$, our thesis follows.
\end{proof}

Notice that, as in the proof of Theorem \ref{conv-in-mu-thm}, this implies that for any $x \in\,H$
\begin{equation}
\label{s48}
\lim_{\e\to 0}\,\sup_{t \in\,[0,T]}|T_\e(t)x-T_0(t)x|_H=0.
\end{equation}

Now, as a consequence of \eqref{Se-to-S0-1}, by proceeding as in the proof of Corollary \ref{cor3.7},  we obtain the following result.
\begin{Lemma}
\label{s34}
  For any $\psi \in L^1(\Omega; L^1([0,T];H))$
  \begin{equation} \label{Se-to-S0-integral}
  \lim_{\e \to 0}\E \sup_{t  \in\,[0,T]} \left| \int_0^t \left( T_\e(t-s)J_\e^{-1}\psi(s) - T_0(t-s)J_0^{-1} \psi(s) \right)ds \right|_H = 0.
  \end{equation}
\end{Lemma}

Now, $u_\e$ is the unique mild solution in $L^p(\Omega;C([0,T];H)$ of problem \eqref{first-abst} (in the case of additive noise) or problem \eqref{s10} (in the case of multiplicative noise), so that
\[u_\e(t)=T_\e(t)u_0+\int_0^t T_{\e}(t-s)B_\e(u_\e(s),s)\,ds+\Gamma_\e(u_\e)(t).\]
Moreover, $u(t)$ is the unique mild solution in $L^p(\Omega;C([0,T];H))$ of the problem
\[du(t)=\le[A_0 u(t)+B_0(u(t),t)\r]\,dt+G_0(u(t),t)\,dw^Q(t),\ \ \ \ u(0)=u_0,\]
with
$G_0=J_{0}^{-1}I$ or $G_0=J_0^{-1} G$, so that
\[u(t)=T_0(t)u_0+\int_0^t T_{0}(t-s)B_0(u_\e(s),s)\,ds+\Gamma_0(u_\e)(t).\]

Then, in view the previous two lemmas, we have that the arguments used in the proof of Theorem \ref{conv-in-mu-thm} and Theorems \ref{thm:Gamma-mu-to-Gamma-e} and \ref{ts31} can be repeated and we have the following result.

\begin{Theorem}
\label{ts56}
Assume either $G$ satisfies Hypothesis \ref{H3} or $G(x,t)=I$. Then, under Hypotheses \ref{H2} and \ref{H6}, we have that for any $T>0$ and $p\geq 1$
\begin{equation}
\label{s35}
\lim_{\e\to 0}\,\E\,|u_\e-u|^p_{C([0,T];H)}=0.
\end{equation}
\end{Theorem}

We conclude this section by showing that the convergence result proved above for $\e\downarrow 0$ is also valid for the second order system, that is for every $\mu>0$ fixed.

\begin{Theorem}
\label{ts55}
Assume either $G$ satisfies Hypothesis \ref{H3} or $G(x,t)=I$. Then, under Hypotheses \ref{H2} and \ref{H6}, we have that
  for any initial conditions $(u_0,v_0)$ and $\mu>0$
  \[\lim_{\e \to 0} \E\,|z_\mu^\e - z_\e|_{C([0,T];\H)}^{2m},\]
  for any $T>0$ and $p\geq 1$.
\end{Theorem}
As long as we can show that $S^\e_\mu(t) P_n z\to S^0_\mu(t) P_n z$ for any fixed $n$, we can prove Theorem \ref{ts55} by following the arguments of Theorems  \ref{conv-in-mu-thm} and \ref{ts31}. Fortunately, we can prove something stronger. We show that $ \sup_{t \geq0}\|S^\e_\mu(t) - S^0_\mu(t)\|_{\mathcal{L}(\H)}=0.$

To this purpose, we introduce an equivalent norm  on $H \times H^{-1}$ by setting
\[|(x,y)|_{\H(\mu)}^2 = |x|_H^2 + \mu|y|_{H^{-1}}^2.\]
Because of \eqref{energy-est-1}, for any $\e\geq0$,
\begin{equation} \label{eq:Smu-contraction}
  \sup_{t \geq 0}\| S_\mu^\e(t)\|_{{\cal L}(\H)} \leq 1.
\end{equation}
Note that if $\e=0$, then, by \eqref{energy-est-1}, for any $z \in \H$ and $t \geq0$,
\[|S_\mu^0(t) z|_{\H(\mu)} = |z|_{\H(\mu)}.\]
\begin{Lemma}
  For fixed $\mu>0$ and  $T>0$,
  \begin{equation} \label{eq:Smue-to-Smu0-strongly}
    \lim_{\e \to 0} \sup_{t \in\,[0,T]} \|S_\mu^\e(t) - S_\mu^0(t) \|_{{\cal L}(\H(\mu))} =0.
  \end{equation}
\end{Lemma}
\begin{proof}
If we fix $z \in \H$ and define
$u^\e_\mu(t) = \Pi_1 S_\mu^\e(t)z$, then
  \[\mu \frac{\partial^2 u_\mu^\e}{\partial t^2} (t) + J_\e \frac{\partial u_\mu^\e}{\partial t}(t) = \Delta u_\mu^\e(t).\]
  Therefore, if we define $\gamma_\mu^\e(t) = u_\mu^\e(t) - u_\mu^0(t)$, we have
  \[\mu \frac{\partial^2 \gamma_\mu^\e}{\partial t^2}(t) + J_\e \frac{\partial u_\mu^\e}{\partial t}(t) - J_0 \frac{\partial u_\mu^0}{ \partial t}(t) = \Delta \gamma_
  \mu^\e(t). \]
Since $J_\e - J_0 = \e I$, we have
  \[J_\e \frac{\partial u_\mu^\e}{\partial t}(t) - J_0 \frac{\partial u_\mu^0}{ \partial t} = J_\e \frac{\partial \gamma_\mu^\e}{\partial t} + \e \frac{\partial u_\mu^0}{\partial t}, \]
so that
  \[\mu \frac{\partial^2 \gamma_\mu^\e}{\partial t^2}(t) + J_\e \frac{\partial \gamma_\mu^\e}{\partial t}(t) = \Delta \gamma_\mu^\e(t) - \e \frac{\partial u_\mu^0}{\partial t}(t),\ \ \ \gamma_\mu^\e(0) = \frac{\partial \gamma_\mu^\e}{\partial t} (0) = 0. \]
This yields
  \begin{equation} \label{eq:Smue-minus-Smu0}
    \begin{array}{l}
      \ds{S_\mu^\e(t)z - S_\mu^0(t)z = \left(\gamma_\mu^\e(t), \frac{\partial \gamma_\mu^\e}{\partial t}(t) \right) = - \frac  \e\mu \int_0^t S_\mu^\e(t-s)\le(0,  \frac{\partial u_\mu^0}{\partial s}(s) \r)ds}\\
      \ds{ = -  \frac{\e}{\mu} \int_0^t S_\mu^\e(t-s) \le(0,S_\mu^0(s)z \r)ds.}
    \end{array}
  \end{equation}
  Then, by \eqref{eq:Smu-contraction}, since  $|(0,\Pi_2 z)|_{\H(\mu)} \leq |z|_{\H(\mu)}$, we conclude
  \[\begin{array}{l}
    \ds{\left| S_\mu^\e(t)z - S_\mu^0(t)z \right|_{\H(\mu)} \leq \frac{\e}{\mu} \int_0^t \left|(0, \Pi_2 S_\mu^0(s) z) \right|_{\H(\mu)}}ds \leq \frac{\e\, t}{\mu} |z|_{\H(\mu)},
  \end{array}\]
and \eqref{eq:Smue-to-Smu0-strongly} follows.
\end{proof}

\end{document}